\documentclass{amsart}
\usepackage[utf8]{inputenc}
\usepackage{bbold}
\usepackage{amsmath}
\usepackage{graphicx, hyperref}
\usepackage{wrapfig}
\usepackage{amsthm}
\usepackage{color}
\usepackage{float}
\usepackage{tikz}
\usetikzlibrary{cd}
\usetikzlibrary{decorations.markings,decorations.pathreplacing}
\usepackage{nicematrix}

\usepackage[utf8]{inputenc}
\usepackage{amsmath, color}
\usepackage{graphicx}
\usepackage{wrapfig}
\usepackage{amsthm}
\usepackage[margin=1in]{geometry}
\newtheorem*{corollary*}{Corollary}

\newtheorem*{question*}{Question}

\definecolor{lightviolet}{RGB}{197,180,227}

\DeclareMathOperator{\h}{\mathcal{H}}

\DeclareMathOperator{\R}{\mathbb{R}}
\DeclareMathOperator{\Z}{\mathbb{Z}}

\DeclareMathOperator{\Ver}{V}
\DeclareMathOperator{\Vol}{\rm Vol}

\DeclareMathOperator{\Edg}{E}
\DeclareMathOperator{\Ext}{{\rm Ext}}
\DeclareMathOperator{\ext}{{\rm ext}}
\newtheorem{maintheorem}{Theorem}

\newtheorem{theorem}{Theorem}[section]

\newtheorem{lemma}[theorem]{Lemma}
\newtheorem{proposition}[theorem]{Proposition}
\newtheorem{definition-theorem}[theorem]{Definition/Theorem}
\newtheorem{corollary}[theorem]{Corollary}

\newtheorem{conjecture}[theorem]{Conjecture}

\theoremstyle{definition}
\newtheorem{example}[theorem]{Example}
\newtheorem{definition}[theorem]{Definition}
\newtheorem{remark}[theorem]{Remark}

\DeclareMathOperator{\conv}{ConvHull}

\title{Trapezodial property of the generalized Alexander polynomial}\date{}

 \author{Tam\'as K\'alm\'an}

\address{Tam\'as K\'alm\'an, Department of Mathematics,
Institute of Science Tokyo, Japan.
\newline\textup{kalman@math.titech.ac.jp}
}

\author{Karola M\'esz\'aros}

\address{Karola M\'esz\'aros, Department of Mathematics, Cornell University, Ithaca, NY 14853. \newline\textup{karola@math.cornell.edu}
}

\author{Alexander Postnikov}

\address{Alexander Postnikov, Department of Mathematics, Massachusetts Institute of Technology, Cambridge, MA 02139. \newline\textup{apost@math.mit.edu}
}

\thanks{Tam\'as K\'alm\'an  was supported by the Japan Society for the Promotion of Science (JSPS) Grant-in-Aid for Scientific Research C no.\ 23K03108. Karola M\'esz\'aros was  supported by NSF Grants ~DMS-1847284 and  ~DMS-2348676. Alexander Postnikov   was  supported by NSF Grant DMS-2054129.}

\begin{document}

\begin{abstract} Fox's conjecture from 1962, that the absolute values of the coefficients of the  Alexander polynomial of an alternating link are trapezoidal, 
has remained stubbornly open to this date. Recently  Fox's conjecture was settled for all special alternating links. In this paper we take a broad view of the Alexander polynomials of special alternating links, showing that they are a generating function for a statistic on certain vector configurations. We  study three types of vector configurations: (1) vectors arising from cographic matroids, (2) vectors arising from graphic matroids, (3) vectors arising from totally positive matrices. We prove that Alexander polynomials of special alternating links belong to both classes (1) and (2), and prove log-concavity, respectively trapezoidal, properties for classes (2) and (3). As a special case of our results, we obtain a new proof of Fox's conjecture for special alternating links.
 \end{abstract}

\maketitle

\section{Introduction}

Log-concavity abounds in mathematics ---  with proofs often  hard. 
Log-concavity of a sequence $a_1, \ldots, a_n$   with no internal zeros implies unimodality, and more strongly, it implies the trapezoidal property: $$a_1<a_2<\cdots <a_k=a_{k+1}=\cdots=a_m>a_{m+1}>\cdots>a_n$$ for some $1\leq k\leq m\leq n$.  
 Fox's famous \textit{trapezoidal conjecture} from 1962 \cite{fox} states that the Alexander polynomial of an alternating link is trapezoidal; see Section \ref{sec:knot} and Conjecture \ref{fox} for more details. This is known to hold in a number of special cases, see \cite{tamas-recent} and references within. In particular, it was recently settled for special alternating links in \cite{hafner2023logconcavity} by proving the stronger property of log-concavity. 

In this paper we consider two natural stronger criteria for a sequence that imply the trapezoidal property. One is that the sequence is log-concave with no internal zeros, as mentioned above. The other is what we refer to as \emph{box-positivity}, saying that the sequence coincides with the sequence of coefficients of a polynomial that we can write as a nonnegative combination of products of $q$-numbers $[n_1]_q\cdots[n_k]_q$, where the sum of the integers $n_1, \ldots, n_k$ is fixed (see Definition \ref{def:box} and Lemma \ref{lem:box}). 
We will meet both scenarios below.  
 
We study the family of polynomials
 $f_{{A}}(t)$, defined in \cite{slicing} for flat matrices $A=[a_{1}^T\mid \cdots \mid a_{N}^T]$, where $a_i \in \R^d$ for $i \in [N]$. All these polynomials are palindromic and we show that they include the Alexander polynomial of any special alternating link as a special case. The polynomial $f_{{A}}(t)$  enumerates the bases of the oriented matroid $M[A]$, associated to a {flat} matrix $A$, according to the statistic $\ext_\rho$ called {external semi-activity}:
 $$ f_{{A}}(t)=\sum_{B} t^{\ext_{\rho}(B)} \Vol{\Pi_B}, $$ where the sum is over all bases $B$ of  $M[A]$, $\Vol$ is a volume form on $\R^d$, and $\Pi_B$ is a parallelepiped associated to $B$.  
Here $\rho$ is a discrete parameter which is necessary to set up the formula, but it does not influence the coefficients of $f_A(t)$. A matrix $A$ is called \emph{flat} if $a_1, \ldots, a_N$ lie in an affine hyperplane $H=\{x \in \R^d \mid h(x)=1\}$, where $h \in {(\R^d)}^*$ is a linear form.  See Section \ref{sec:f_A} for a detailed explanation and precise definitions.

In this paper we study the polynomial   
$f_{{A}}(t)$  for flat matrices $A$ whose columns constitute one of the following three types of vector configurations:
\begin{enumerate}
\item\label{item:cograph} 
vectors arising from cographic matroids, 
\item\label{item:graph} vectors arising from graphic matroids, 
\item vectors arising from totally positive matrices. 
\end{enumerate}

 The first of our main results is that the  Alexander polynomial of a special alternating link is a special case of $f_{{A}}(t)$, as it belongs to both classes \eqref{item:cograph} and \eqref{item:graph}:

 \begin{maintheorem} \label{thm:intro-Alex} The Alexander polynomial $\Delta_L(t)$, associated to a special alternating link $L$, equals  $f_{{A}}(-t)$ for a flat matrix $A$, whose oriented matroid $M[A]$ is both graphic and cographic. 
 \end{maintheorem}
 
 See Theorem \ref{cor:bip} in the text for the above theorem, as well as  Section \ref{sec:knot} for knot theory  background on the Alexander polynomial and special alternating links. 
 Using Theorem \ref{thm:intro-Alex} we can reinterpret the main result of \cite{hafner2023logconcavity} about the Alexander polynomials of special alternating links as follows: the sequence of coefficients of $f_{{A}}(t)$, where $A$ arises from a planar bipartite graph, is log-concave. In Section \ref{sec:graphic} we generalize this for all bipartite graphs --- thus dropping the planarity requirement ---  as follows:

 \begin{maintheorem} \label{thm:intro-graphical} The sequence of coefficients of the polynomial  $f_{{A}}(t)$, associated to the matrix $A$ arising from a bipartite graph, is log-concave. In particular, it is trapezoidal.
 \end{maintheorem}
 
 See Theorem \ref{thm:graphical} in the text for the above theorem. 
 Note that as a consequence of Theorems \ref{thm:intro-Alex} and \ref{thm:intro-graphical} we obtain a new proof of Fox's conjecture for special alternating links:
 
 \begin{corollary*} 
   \cite[Theorem 1.2]{hafner2023logconcavity}
 The coefficients of the Alexander polynomial 
 $\Delta_L(-t)$, where $L$ is a special alternating link, form a  log-concave sequence with no internal zeros. In particular the sequence is trapezoidal, proving Fox's conjecture in this case.
\end{corollary*}

This result appears as Corollary \ref{cor:main} in the text. 
 Finally, we turn to totally positive matrices $A$:
 
  \begin{maintheorem} \label{thm:intro-totpos}
  The sequence of coefficients of the polynomial  $f_{{A}}(t)$, associated to
  a flat totally positive matrix with last row all $1$s, is box-positive. In particular, it is trapezoidal.  
   \end{maintheorem}
   
   See Section \ref{sec:TP} for more details. Theorem \ref{thm:intro-totpos} appears as Corollary \ref{cor:C} in the text. 
With Theorems \ref{thm:intro-graphical} and \ref{thm:intro-totpos}  as our motivation, we pose the  following broad question:

   \begin{question*} Is it true that the sequence of coefficients of the polynomial  $f_{{A}}(t)$, associated to a flat matrix $A$, is always trapezoidal? When do the stronger properties of log-concavity and/or box-positivity hold?
    \end{question*}
       

\noindent \textbf{Roadmap of the paper.}  In Section \ref{sec:knot} we explain the knot theory background necessary for the paper.  In Section \ref{sec:f_A} we introduce the polynomials $f_{{A}}(t)$ that are our main objects of study. In Section \ref{sec:graph-cograph} we investigate when graphic and cographic matroids yield  matrices $A$ for which $f_A$ is defined. In Section \ref{sec:cographic} we show that the Alexander polynomial of a special alternating link is a special case of the polynomials $f_{{A}}(t)$ as stated in Theorem \ref{thm:intro-Alex}. In Section \ref{sec:intpoint} we give an alternative description of the polynomials $f_{{A}}(t)$ arising from matrices $A$ whose columns form a unimodular vector configuration. Section \ref{sec:graphic} is devoted to proving Theorem \ref{thm:intro-graphical}, while Section \ref{sec:TP} contains the proof of  Theorem \ref{thm:intro-totpos}.

\section{Knot theory background}
\label{sec:knot}

Discovered in the 1920's \cite{alexander1928topological}, the Alexander polynomial $\Delta_L(t)$, associated to an oriented link $L$, was the first polynomial-valued knot invariant. More precisely,  $\Delta_L(t)\in\Z[t]$ and invariance means that if the oriented links $L_1$ and $L_2$ are isotopic then $\Delta_{L_1}(t)\sim\Delta_{L_2}(t)$, where $\sim$ denotes equality up to multiplication by $\pm t^k$ for some integer~$k$.  
The present paper is concerned with  special alternating links only, and as such we will define the Alexander polynomial only for these; see Section \ref{sub:alexpoly} for the definition.

\subsection{Fox's trapezoidal conjecture} While a lot is known  about the Alexander polynomials of links and its coefficients --- for example, that it is palindromic for all links ---  a lot also remains an intrigue. One of the inspirations for this project is Fox's outstanding conjecture (1962): 
 
 \begin{conjecture} \label{fox} \cite{fox} Let $L$ be an alternating link. Then the coefficients of $\Delta_L(-t)$ form a trapezoidal sequence.
\end{conjecture}

Recall that 
a link is called alternating if it possesses a planar diagram so that when we trace any strand in it, over- and undercrossing sites alternate. A non-diagrammatic definition of this class of links, based on spanning surfaces, is also possible but it has only been found relatively recently \cite{greene,howie}. The knot theory of alternating links is known to lend itself well to investigations by combinatorial methods. For example, the signs of the coefficients of their Alexander polynomials alternate, which is not true in general. Thus, by using the ambiguity mentioned above, we may assume that $\Delta_L(-t)$, for the alternating link $L$, has only non-negative coefficients. We have already done so in the phrasing of Conjecture \ref{fox}.

Fox's conjecture remains open in general, although some special cases have been settled by Hartley \cite{H79} for two-bridge knots, by Murasugi \cite{murasugi} for a family of alternating algebraic links, by Ozsv\'ath and Szab\'o \cite{OS03}, and later by Jong \cite{jong2009alexander}, for genus $2$ alternating knots, by Hafner, M\'esz\'aros and Vidinas  for  special alternating links \cite{hafner2023logconcavity}, and by Azarpendar, Juh\'asz, and K\'alm\'an for certain plumbings of special alternating links \cite{tamas-recent}. See the latter article for a thorough summary on the progress made on Fox's conjecture since ~1962.
 
We note that Stoimenow  \cite{stoi} strengthened  Fox's conjecture to log-concavity with no internal zeros, and some of the above results, including the case of special alternating links \cite{hafner2023logconcavity}, establish this stronger condition.

\subsection{Special alternating links} We now recall the construction of a special alternating link $L_G$ associated to a connected plane bipartite graph $G$, as it appears in Definition/Theorem \ref{alexpoly}, and then we construct the Alexander polynomial of $L_G$.
Given a plane bipartite graph $G$, let $M(G)$ be the medial graph of $G$: the vertex set of $M(G)$ is the set $\{v_e \mid e\in\Edg(G)\}$, and two vertices $v_e$ and $v_{e'}$ of $M(G)$ are connected by an edge if the edges $e$ and $e'$ are consecutive in the boundary of a face of $G$. (Sometimes there are two such faces, resulting in a pair of parallel edges in $M(G)$.) We think of a particular planar drawing of $M(G)$ here: the midpoints of the edges of the planar drawing of $G$ are the vertices of $M(G)$, whence faces of $M(G)$ correspond both to faces and vertices of $G$. See Figure \ref{fig:L_G} for an example.

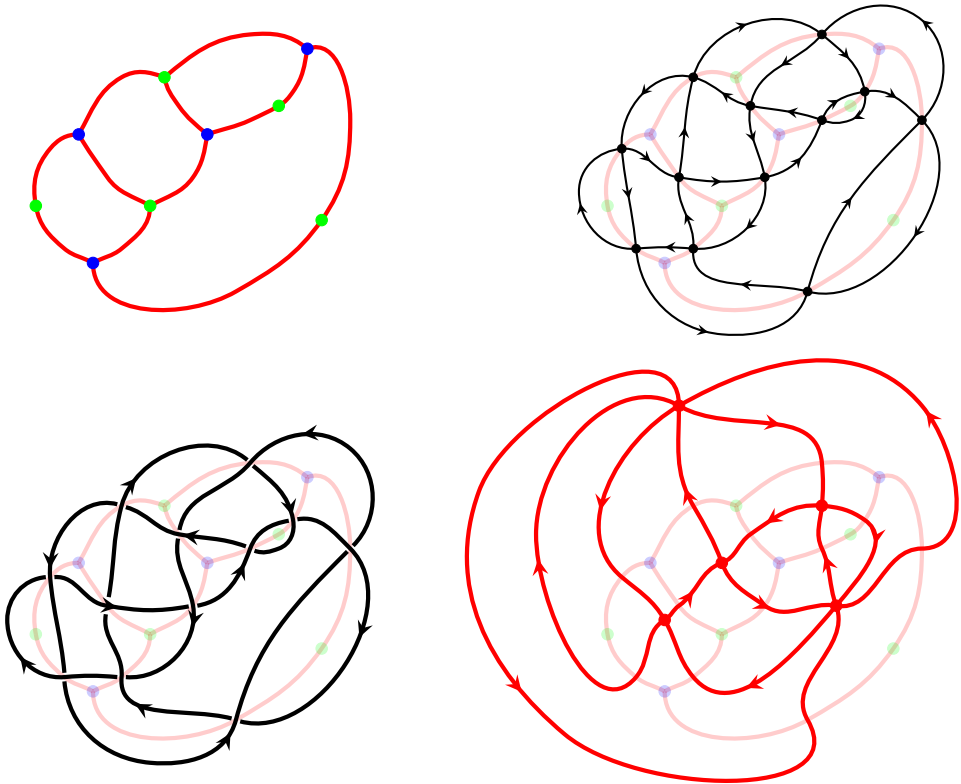
\begin{figure}[h]
\begin{tikzpicture}[scale=.19,decoration={markings, 
    mark= at position 0.5 with {\arrow{stealth}}}
] 

\begin{scope}[shift={(-40,0)}]
\draw [ultra thick, red] (-4,1) to [out=100,in=-130] (-3,5) to [out=50,in=180] (-1,6);
\draw [ultra thick, red] (-4,1) to [out=-80,in=140] (-2,-2) to [out=-40,in=150] (0,-3);
\draw [ultra thick, red] (0,-3) to [out=30,in=-140] (2,-2) to [out=40,in=-90] (4,1);
\draw [ultra thick, red] (4,1) to [out=150,in=-55] (1,3) to [out=125,in=-60] (-1,6);
\draw [ultra thick, red] (4,1) to [out=30,in=-125] (7,3) to [out=55,in=-100] (8,6);
\draw [ultra thick, red] (8,6) to [out=140,in=-55] (6,8) to [out=125,in=-80] (5,10);
\draw [ultra thick, red] (-1,6) to [out=60,in=-150] (2,10) to [out=30,in=160] (5,10);
\draw [ultra thick, red] (5,10) to [out=40,in=-175] (11,13) to [out=5,in=140] (15,12);
\draw [ultra thick, red] (8,6) to [out=20,in=-155] (11,7) to [out=25,in=-150] (13,8);
\draw [ultra thick, red] (13,8) to [out=30,in=-125] (14,9) to [out=55,in=-100] (15,12);
\draw [ultra thick, red] (15,12) to [out=20,in=90] (18,7) to [out=-90,in=55] (16,0);
\draw [ultra thick, red] (16,0) to [out=-125,in=30] (10,-5) to [out=-150,in=-90] (0,-3);

\draw [fill=blue,blue] (0, -3) circle [radius=0.4];
\draw [fill=blue,blue] (15, 12) circle [radius=0.4];
\draw [fill=blue,blue] (-1, 6) circle [radius=0.4];
\draw [fill=blue,blue] (8, 6) circle [radius=0.4];
\draw [fill=green,green] (4, 1) circle [radius=0.4];
\draw [fill=green,green] (-4, 1) circle [radius=0.4];
\draw [fill=green,green] (5, 10) circle [radius=0.4];
\draw [fill=green,green] (13, 8) circle [radius=0.4];
\draw [fill=green,green] (16, 0) circle [radius=0.4];
\end{scope}

\draw [opacity=.2,ultra thick, red] (-4,1) to [out=100,in=-130] (-3,5) to [out=50,in=180] (-1,6);
\draw [opacity=.2,ultra thick, red] (-4,1) to [out=-80,in=140] (-2,-2) to [out=-40,in=150] (0,-3);
\draw [opacity=.2,ultra thick, red] (0,-3) to [out=30,in=-140] (2,-2) to [out=40,in=-90] (4,1);
\draw [opacity=.2,ultra thick, red] (4,1) to [out=150,in=-55] (1,3) to [out=125,in=-60] (-1,6);
\draw [opacity=.2,ultra thick, red] (4,1) to [out=30,in=-125] (7,3) to [out=55,in=-100] (8,6);
\draw [opacity=.2,ultra thick, red] (8,6) to [out=140,in=-55] (6,8) to [out=125,in=-80] (5,10);
\draw [opacity=.2,ultra thick, red] (-1,6) to [out=60,in=-150] (2,10) to [out=30,in=160] (5,10);
\draw [opacity=.2,ultra thick, red] (5,10) to [out=40,in=-175] (11,13) to [out=5,in=140] (15,12);
\draw [opacity=.2,ultra thick, red] (8,6) to [out=20,in=-155] (11,7) to [out=25,in=-150] (13,8);
\draw [opacity=.2,ultra thick, red] (13,8) to [out=30,in=-125] (14,9) to [out=55,in=-100] (15,12);
\draw [opacity=.2,ultra thick, red] (15,12) to [out=20,in=90] (18,7) to [out=-90,in=55] (16,0);
\draw [opacity=.2,ultra thick, red] (16,0) to [out=-125,in=30] (10,-5) to [out=-150,in=-90] (0,-3);

\draw [opacity=.2,fill=blue,blue] (0, -3) circle [radius=0.4];
\draw [opacity=.2,fill=blue,blue] (15, 12) circle [radius=0.4];
\draw [opacity=.2,fill=blue,blue] (-1, 6) circle [radius=0.4];
\draw [opacity=.2,fill=blue,blue] (8, 6) circle [radius=0.4];
\draw [opacity=.2,fill=green,green] (4, 1) circle [radius=0.4];
\draw [opacity=.2,fill=green,green] (-4, 1) circle [radius=0.4];
\draw [opacity=.2,fill=green,green] (5, 10) circle [radius=0.4];
\draw [opacity=.2,fill=green,green] (13, 8) circle [radius=0.4];
\draw [opacity=.2,fill=green,green] (16, 0) circle [radius=0.4];

\draw [fill] (-3,5) circle [radius=.3];
\draw [fill] (-2,-2) circle [radius=.3];
\draw [fill] (2,-2) circle [radius=.3];
\draw [fill] (1,3) circle [radius=.3];
\draw [fill] (7,3) circle [radius=.3];
\draw [fill] (6,8) circle [radius=.3];
\draw [fill] (2,10) circle [radius=.3];
\draw [fill] (11,13) circle [radius=.3];
\draw [fill] (11,7) circle [radius=.3];
\draw [fill] (14,9) circle [radius=.3];
\draw [fill] (18,7) circle [radius=.3];
\draw [fill] (10,-5) circle [radius=.3];

\draw [thick,postaction={decorate}] (7,3) to [out=10,in=-110] (11,7);
\draw [thick,postaction={decorate}] (11,7) to [out=70,in=-170] (14,9);
\draw [thick,postaction={decorate}] (14,9) to [out=10,in=135] (18,7);
\draw [thick,postaction={decorate}] (18,7) to [out=-45,in=-15] (10,-5);
\draw [thick,postaction={decorate}] (10,-5) to [out=165,in=-95] (2,-2);
\draw [thick,postaction={decorate}] (2,-2) to [out=85,in=-100] (1,3);
\draw [thick,postaction={decorate}] (1,3) to [out=80,in=-105] (2,10);
\draw [thick,postaction={decorate}] (2,10) to [out=75,in=140] (11,13);
\draw [thick,postaction={decorate}] (11,13) to [out=-40,in=100] (14,9);
\draw [thick,postaction={decorate}] (14,9) to [out=-80,in=20] (13,7) to [out=-160,in=-20] (11,7);
\draw [thick,postaction={decorate}] (11,7) to [out=160,in=-10] (6,8);
\draw [thick,postaction={decorate}] (6,8) to [out=170,in=-15] (2,10);
\draw [thick,postaction={decorate}] (2,10) to [out=165,in=95] (-3,5);
\draw [thick,postaction={decorate}] (-3,5) to [out=-85,in=95] (-2,-2);
\draw [thick,postaction={decorate}] (-2,-2) to [out=-85,in=175] (4,-8) to [out=-5,in=-105] (10,-5);
\draw [thick,postaction={decorate}] (10,-5) to [out=75,in=-135] (18,7);
\draw [thick,postaction={decorate}] (18,7) to [out=45,in=-40] (18,14) to [out=140,in=50] (11,13);
\draw [thick,postaction={decorate}] (11,13) to [out=-130,in=80] (6,8);
\draw [thick,postaction={decorate}] (6,8) to [out=-100,in=100] (7,3);
\draw [thick,postaction={decorate}] (7,3) to [out=-80,in=-5] (2,-2);
\draw [thick,postaction={decorate}] (2,-2) to [out=175,in=5] (-2,-2);
\draw [thick,postaction={decorate}] (-2,-2) to [out=-175,in=-90] (-6,2) to [out=90,in=-175] (-3,5);
\draw [thick,postaction={decorate}] (-3,5) to [out=5,in=175] (1,3);
\draw [thick,postaction={decorate}] (1,3) to [out=-10,in=-170] (7,3);

\begin{scope}[shift={(-40,-30)}]
\draw [opacity=.2,ultra thick, red] (-4,1) to [out=100,in=-130] (-3,5) to [out=50,in=180] (-1,6);
\draw [opacity=.2,ultra thick, red] (-4,1) to [out=-80,in=140] (-2,-2) to [out=-40,in=150] (0,-3);
\draw [opacity=.2,ultra thick, red] (0,-3) to [out=30,in=-140] (2,-2) to [out=40,in=-90] (4,1);
\draw [opacity=.2,ultra thick, red] (4,1) to [out=150,in=-55] (1,3) to [out=125,in=-60] (-1,6);
\draw [opacity=.2,ultra thick, red] (4,1) to [out=30,in=-125] (7,3) to [out=55,in=-100] (8,6);
\draw [opacity=.2,ultra thick, red] (8,6) to [out=140,in=-55] (6,8) to [out=125,in=-80] (5,10);
\draw [opacity=.2,ultra thick, red] (-1,6) to [out=60,in=-150] (2,10) to [out=30,in=160] (5,10);
\draw [opacity=.2,ultra thick, red] (5,10) to [out=40,in=-175] (11,13) to [out=5,in=140] (15,12);
\draw [opacity=.2,ultra thick, red] (8,6) to [out=20,in=-155] (11,7) to [out=25,in=-150] (13,8);
\draw [opacity=.2,ultra thick, red] (13,8) to [out=30,in=-125] (14,9) to [out=55,in=-100] (15,12);
\draw [opacity=.2,ultra thick, red] (15,12) to [out=20,in=90] (18,7) to [out=-90,in=55] (16,0);
\draw [opacity=.2,ultra thick, red] (16,0) to [out=-125,in=30] (10,-5) to [out=-150,in=-90] (0,-3);

\draw [opacity=.2,fill=blue,blue] (0, -3) circle [radius=0.4];
\draw [opacity=.2,fill=blue,blue] (15, 12) circle [radius=0.4];
\draw [opacity=.2,fill=blue,blue] (-1, 6) circle [radius=0.4];
\draw [opacity=.2,fill=blue,blue] (8, 6) circle [radius=0.4];
\draw [opacity=.2,fill=green,green] (4, 1) circle [radius=0.4];
\draw [opacity=.2,fill=green,green] (-4, 1) circle [radius=0.4];
\draw [opacity=.2,fill=green,green] (5, 10) circle [radius=0.4];
\draw [opacity=.2,fill=green,green] (13, 8) circle [radius=0.4];
\draw [opacity=.2,fill=green,green] (16, 0) circle [radius=0.4];

\draw [ultra thick,postaction={decorate}] (7.3,3) to [out=10,in=-110] (11,7) to [out=70,in=-170] (13.7,8.95);
\draw [ultra thick,postaction={decorate}] (14.3,9.05) to [out=10,in=135] (18,7) to [out=-45,in=-15] (10.3,-5.1);
\draw [ultra thick,postaction={decorate}] (9.7,-4.9) to [out=165,in=-95] (2,-2) to [out=85,in=-100] (.9,2.4);
\draw [ultra thick,postaction={decorate}] (1.1,3.6) to [out=80,in=-105] (2,10) to [out=75,in=140] (10.8,13.2);
\draw [ultra thick,postaction={decorate}] (11.2,12.8) to [out=-40,in=100] (14,9) to [out=-80,in=20] (13,7) to [out=-160,in=-20] (11.3,6.85);
\draw [ultra thick,postaction={decorate}] (10.7,7.15) to [out=160,in=-10] (6,8) to [out=170,in=-15] (2.3,9.9);
\draw [ultra thick,postaction={decorate}] (1.7,10.1) to [out=165,in=95] (-3,5) to [out=-85,in=95] (-2,-1.7);
\draw [ultra thick,postaction={decorate}] (-2,-2.3) to [out=-85,in=175] (4,-8) to [out=-5,in=-105] (10,-5) to [out=75,in=-135] (17.8,6.8);
\draw [ultra thick,postaction={decorate}] (18.2,7.2) to [out=45,in=-40] (18,14) to [out=140,in=50] (11,13) to [out=-130,in=80] (6.05,8.3);
\draw [ultra thick,postaction={decorate}] (5.95,7.7) to [out=-100,in=100] (7,3) to [out=-80,in=-5] (2.3,-2);
\draw [ultra thick,postaction={decorate}] (1.7,-2) to [out=175,in=5] (-2,-2) to [out=-175,in=-90] (-6,2) to [out=90,in=-175] (-3.3,5);
\draw [ultra thick,postaction={decorate}] (-2.7,5) to [out=5,in=175] (1,3) to [out=-10,in=-170] (6.7,2.95);
\end{scope}

\begin{scope}[shift={(0,-30)}]
\draw [opacity=.2,ultra thick, red] (-4,1) to [out=100,in=-130] (-3,5) to [out=50,in=180] (-1,6);
\draw [opacity=.2,ultra thick, red] (-4,1) to [out=-80,in=140] (-2,-2) to [out=-40,in=150] (0,-3);
\draw [opacity=.2,ultra thick, red] (0,-3) to [out=30,in=-140] (2,-2) to [out=40,in=-90] (4,1);
\draw [opacity=.2,ultra thick, red] (4,1) to [out=150,in=-55] (1,3) to [out=125,in=-60] (-1,6);
\draw [opacity=.2,ultra thick, red] (4,1) to [out=30,in=-125] (7,3) to [out=55,in=-100] (8,6);
\draw [opacity=.2,ultra thick, red] (8,6) to [out=140,in=-55] (6,8) to [out=125,in=-80] (5,10);
\draw [opacity=.2,ultra thick, red] (-1,6) to [out=60,in=-150] (2,10) to [out=30,in=160] (5,10);
\draw [opacity=.2,ultra thick, red] (5,10) to [out=40,in=-175] (11,13) to [out=5,in=140] (15,12);
\draw [opacity=.2,ultra thick, red] (8,6) to [out=20,in=-155] (11,7) to [out=25,in=-150] (13,8);
\draw [opacity=.2,ultra thick, red] (13,8) to [out=30,in=-125] (14,9) to [out=55,in=-100] (15,12);
\draw [opacity=.2,ultra thick, red] (15,12) to [out=20,in=90] (18,7) to [out=-90,in=55] (16,0);
\draw [opacity=.2,ultra thick, red] (16,0) to [out=-125,in=30] (10,-5) to [out=-150,in=-90] (0,-3);

\draw [opacity=.2,fill=blue,blue] (0, -3) circle [radius=0.4];
\draw [opacity=.2,fill=blue,blue] (15, 12) circle [radius=0.4];
\draw [opacity=.2,fill=blue,blue] (-1, 6) circle [radius=0.4];
\draw [opacity=.2,fill=blue,blue] (8, 6) circle [radius=0.4];
\draw [opacity=.2,fill=green,green] (4, 1) circle [radius=0.4];
\draw [opacity=.2,fill=green,green] (-4, 1) circle [radius=0.4];
\draw [opacity=.2,fill=green,green] (5, 10) circle [radius=0.4];
\draw [opacity=.2,fill=green,green] (13, 8) circle [radius=0.4];
\draw [opacity=.2,fill=green,green] (16, 0) circle [radius=0.4];

\draw [red,ultra thick,postaction={decorate}] (4,6) to [out=110,in=-60] (2,10) to [out=120,in=-90] (1,17);
\draw [red,ultra thick,postaction={decorate}] (4,6) to [out=-70,in=145] (7,3) to [out=-35,in=170] (12,3);
\draw [red,ultra thick,postaction={decorate}] (11,10) to [out=175,in=35] (6,8) to [out=-145,in=20] (4,6);
\draw [red,ultra thick,postaction={decorate}] (12,3) to [out=110,in=-65] (11,7) to [out=115,in=-95] (11,10);
\draw [red,ultra thick,postaction={decorate}] (0,2) to [out=30,in=-145] (1,3) to [out=35,in=-160] (4,6);
\draw [red,ultra thick,postaction={decorate}] (1,17) to [out=-150,in=140] (-3,5) to [out=-40,in=120] (0,2);
\draw [red,ultra thick,postaction={decorate}] (12,3) to [out=-130,in=-50] (2,-2) to [out=130,in=-60] (0,2);
\draw [red,ultra thick,postaction={decorate}] (11,10) to [out=-5,in=145] (14,9) to [out=-35,in=50] (12,3);
\draw [red,ultra thick,postaction={decorate}] (1,17) to [out=-30,in=95] (11,13) to [out=-85,in=85] (11,10);
\draw [red,ultra thick,postaction={decorate}] (0,2) to [out=-150,in=50] (-2,-2) to [out=-130,in=-90] (-9,8) to [out=90,in=150] (1,17);
\draw [red,ultra thick,postaction={decorate}] (12,3) to [out=-10,in=180] (18,7) to [out=0,in=-50] (18,17) to [out=130,in=30] (1,17);
\draw [red,ultra thick,postaction={decorate}] (1,17) to [out=90,in=70] (-13,11) to [out=-110,in=140] (-7,-6) to [out=-40,in=-60] (10,-5) to [out=120,in=-70] (12,3);

\draw [fill=red,red] (1,17) circle [radius=0.4];
\draw [fill=red,red] (12,3) circle [radius=0.4];
\draw [fill=red,red] (0,2) circle [radius=0.4];
\draw [fill=red,red] (4,6) circle [radius=0.4];
\draw [fill=red,red] (11,10) circle [radius=0.4];
\end{scope}
\end{tikzpicture}
\caption{A plane bipartite graph $G$, medial graph $M(G)$, the associated special alternating link $L_G$, and the dual graph $G^*$ of $G$ endowed with an Eulerian 
orientation that makes it into an alternating dimap $D$.}
\label{fig:L_G}
\end{figure}

The edges of $M(G)$ also receive orientations by stipulating that the faces of $M(G)$ that surround vertices of $G$ be consistently oriented, while around the boundaries of the other faces of $M(G)$ the orientations alternate.
Notice that this can be achieved (in fact, for a connected $G$, in exactly two ways, which differ by an overall change of orientation) 
because $G$ is bipartite and thus its faces are bounded by an even number of edges.

Thinking of $M(G)$ as a flattening of an oriented link, there are two ways to choose under- and overcrossings at each vertex of $M(G)$ to make it into a diagram of an alternating link $L_G$. 
One of the choices, the one we will prefer, is such that 
each crossing is \emph{positive}, meaning that as we travel along the overcrossing strand, we see the undercrossing strand pass from right to left. Thus our procedure yields a positive special alternating link. Moreover, any non-split positive special alternating link arises from such a construction.

\subsection{The Alexander polynomial in the case of special alternating links} \label{sub:alexpoly} We use a theorem of Murasugi and Stoimenow \cite{even} as the definition of the Alexander polynomial of a special alternating link, as this is the most relevant for the purposes of this paper. We refer the reader to \cite{rolfsen} for a general definition. 

Special alternating links $L_G$ are associated to plane bipartite graphs $G$ as above. Equivalently, they are associated to the planar dual $G^*$ of 
$G$.   
 The planar dual of a bipartite graph is always Eulerian, and moreover, it can be endowed with a natural orientation: whenever an edge of $G^*$ crosses an edge of $G$, we make a convention to keep a fixed color class on our right. Indeed, such an orientation makes $G^*$ into an alternating dimap $D$. Here an \textbf{alternating dimap} is a plane Eulerian digraph, where at each vertex the edges alternate coming in and going out as we move 
 around the vertex. 
See Figure \ref{fig:L_G} for an example. 
 
 Given an Eulerian 
 digraph $D$ and a vertex $r\in\Ver(D)$, a spanning tree of $D$ is a \textbf{$k$-spanning tree} rooted at $r$ if reversing the orientations of exactly $k$ edges produces an oriented spanning tree rooted at $r$. Recall that an \textbf{oriented 
 spanning 
 tree} rooted at $r$ is a subgraph
such that for each vertex $v$, there is a unique path from $v$ to $r$, and this path is directed.  Denote the number of $k$-spanning trees of $D$ rooted at $r$ by $c_k(D,r)$. Define 
 \begin{equation} \label{pdt} P_D(t)=\sum_{k=0}^{\infty} c_k(D, r)\,t^k.\end{equation} 
This polynomial was first introduced by Murasugi and Stoimenow \cite{even}, who showed that for any fixed  Eulerian digraph $D$, which is not necessarily planar, the polynomial $P_D(t)$  does not depend on the root $r$. Moreover, 
 they 
 expressed the Alexander polynomial of a special alternating link purely in terms of alternating dimaps $D$, see Definition/Theorem \ref{alexpoly} below.  This point of view was further investigated in \cite{paper2} from a combinatorial and discrete geometric perspective.  
 
\begin{definition-theorem}\cite[Theorem 2]{even} \label{alexpoly} For an alternating dimap $D$, the polynomial   $P_D(t)$ equals the Alexander polynomial $\Delta_{L_{G}}(-t)$ for the special alternating link  $L_{G}$ associated to the planar dual $G$ of $D$. In other words:
$$\Delta_{L_{G}}(-t) \sim P_D(t).$$
\end{definition-theorem}

\begin{example}
\label{ex:12a_1097}
The Alexander polynomial of the link (in fact, knot) of Figure \ref{fig:L_G} is 
\[\Delta_{L_{G}}(t)=16-54t+77t^2-54t^3+16t^4.\]
It is log-concave as $16\cdot77<54^2<77^2$. From the coefficients we may read off, for instance, that the directed graph $D$ (as well as the bipartite graph $G$) of Figure \ref{fig:L_G} has $217$ spanning trees, among which, with respect to any fixed root, exactly $77$ consist of $2$ incoming and $2$ outgoing edges.
\end{example}

We take Definition/Theorem \ref{alexpoly} as our \textit{definition of the Alexander polynomial} $\Delta_{L_{G}}(t)$ and relate it to the polynomials  $f_{{A}}(t)$ arising from matrices $A$ in Section \ref{sec:cographic}. First however, we need to define  $f_{{A}}(t)$.

\section{Polynomials 
arising from matrices
}
 \label{sec:f_A}

In this section we define our main object of study in this paper, namely the  polynomial $f_{A}$ arising from matrices $A$. We also lay out other necessary concepts for this paper.  The exposition in this section follows the  work   \cite{slicing}. 
 
 \subsection{Oriented matroids arising from matrices} \label{sec:h}
 
  A matrix $A=[a_1^T\mid \cdots\mid a_N^T]$  with   vectors $a_i \in V\cong \R^d$, $i \in [N]$,  is \textbf{flat} if there exists a  
  linear form $h \in V^*$ such that $h(a_i)=1$ for all $i \in [N]$.  A vector arrangement $(a_1, \ldots, a_N)$ 
  in $V$ is \textbf{full dimensional} if it spans $V$.   
 \medskip
 
\noindent \textbf{Assumptions throughout the paper:} The matrix $A=[a_1^T\mid \cdots\mid a_N^T]$  with  vectors $a_i \in V\cong \R^d$, $i \in [N]$,  will be assumed to be   flat and full dimensional in definitions and statements, unless otherwise stated. The    linear form in $V^*$ witnessing the flatness of $A$ will be denoted by $h$ throughout. That is, we will assume that $h(a_i)=1$ for all $i \in [N]$.

 \medskip

 A \textbf{matroid} ${M}$ is an ordered pair $(E, \mathcal{B})$ consisting of a finite set $E$ and a nonempty collection of subsets $\mathcal{B}$ of $E$  such that  
  if $B_1, B_2 \in \mathcal{B}$ and $x \in B_1\setminus B_2$, then there exists    $y \in B_2\setminus B_1$,   such that $(B_1 \cup \{y\}) \setminus \{x\}\in \mathcal{B}.$
 The set $E$ is the \textbf{ground set} of the matroid $M$, while $\mathcal{B}$ is the set of its \textbf{bases}.  It can be shown that the bases of a matroid have the same cardinality.  Bases and all their subsets are called \textbf{independent sets}. The subsets of $E$ that are not  independent are called \textbf{dependent}. A minimal by inclusion dependent set is a \textbf{circuit}. 

In this paper our focus is on oriented matroids arising from matrices $A$. 
 The   \textbf{matroid} $M[A]$ associated to  $A$ is defined to be on the ground set $[N]$. Its bases are the sets $\{i_1, \ldots, i_d\} \subset [N]$ such that  the set of vectors $\{a_{i_1}, \ldots, a_{i_d}\}$ is a linear basis of $\R^d$. For any circuit $C$ of the matroid $M[A]$ there is a unique, up to rescaling, linear dependence \begin{equation} \label{C} \sum_{i \in C} \lambda_i a_i=0,\end{equation} with $\lambda_i\neq 0$ for all $i \in C$.  Set $\lambda_i=0$ for $i \in [N]\setminus C$. Collecting the values of $\lambda_i$, which depend on $C$, we define the vector $\lambda_C=(\lambda_i)_{i \in [N]} \in \R^N$. 
 This is determined up to rescaling by a non-zero real factor. 
 
 For every circuit $C$ of $M[A]$, choose a linear dependence as in equation \eqref{C}. Once we fix this linear dependence, we obtain an orientation of each circuit $C$. Namely, an \textbf{oriented circuit} $C$ is an ordered subdivision of $C$ into $C=C^+ \sqcup C^-$, where
 \begin{equation} C^+=\{i \in C\mid \lambda_i>0\} \text{ and }  C^-=\{i \in C\mid \lambda_i<0\}. 
 \end{equation}
  
  By orienting each circuit of the matroid $M[A]$ as above, it becomes an \textbf{oriented matroid}. Note that our choice of orientation of each circuit $C$ depended on \eqref{C} only up to a switch between $ C^+$ and  $C^-$.
  
  
   \subsection{External semi-activity of the bases of an oriented matroid}

  The \textbf{secondary hyperplane arrangement} $\h_{A}^{sec}=\{H_C\}_{C \text{ a circuit of } M[A]}$ is the arrangement of hyperplanes $H_C$ in $\R^N$, defined for each circuit $C$ of $M[A]$ by 
  \begin{equation} H_C=\{x \in \R^N \mid (\lambda_C, x)=0\}.\end{equation}
  A vector $x \in \R^N$ is called \textbf{generic} (with respect to $\h_{A}^{sec}$) if $x$ lies in the interior of a top dimensional face (so called \emph{chamber}) of the arrangement $\h_{A}^{sec}$ --- in other words, if $x$ is on none of the hyperplanes.

  We can use a generic vector 
  with respect to  $\h_{A}^{sec}$ to induce an orientation of $M[A]$.  
 Namely, such a generic vector $\rho \in  \R^N$ determines the choice of $\lambda_C$, this time up to rescaling by a \emph{positive} real factor, via  requiring that $(\lambda_C, \rho)>0$ for every circuit $C$  of $M[A]$. Note that  $(\lambda_C, \rho)\neq 0$ since $\rho$ is generic, and thus, we can always pick a $\lambda_C$ such that  $(\lambda_C, \rho)>0$ for every circuit $C$.  
 Changing $\rho$ to $\rho'$, where $\rho$ and $\rho'$ are in adjacent top dimensional faces of $\h_{A}^{sec}$ separated by a hyperplane $H_C$, corresponds to flipping the orientation of the circuit $C$, that is, to multiplying \eqref{C} by a negative scalar for that circuit $C$. Note here that a given chamber and a given hyperplane of $\h_{A}^{sec}$ are not necessarily incident (along an $(N-1)$-dimensional polyhedron), that is, if we wish to flip the orientation of exactly one circuit, our choices are typically limited.
  
  \begin{definition} \label{def:ext} Fix a generic vector $\rho \in \R^N$  with respect to $\h_{A}^{sec}$. It determines the orientation of each circuit $C$ of $M[A]$ by choosing $\lambda_C$ such that  $(\lambda_C, \rho)>0$. For a base $B \in M[A]$, we say that $j \in [N]\setminus B$ is \textbf{externally semi-active} with respect to the base $B$ if for the unique circuit $C= C^+\sqcup C^- \subset B\cup \{j\}$, we have $j \in C^+$.
Denote by $\Ext_{\rho}(B)$ the set of externally semi-active elements $j \in [N]\setminus B$ with respect to the base $B$.  Let \begin{equation} \ext_{\rho}(B)=|\Ext_{\rho}(B)|\end{equation} be the \textbf{external semi-activity} of $B$.
\end{definition}

\begin{remark}\label{rem:order}
    A special case of the above construction is obtained by
choosing a linear order on the ground set $[N]$ (say, the order $1 < 2
<  \cdots < N$)
and picking the vector $\rho = (\epsilon, \epsilon^2, \dots, \epsilon^N)$,
where $\epsilon>0$ is a sufficiently small number.   In this case,
each circuit $C$
is oriented as $C = C^+ \sqcup C^{-}$ so that $C^+$ contains the
minimal element of $C$.
For a base $B$, the element $j\in [N]\setminus B$ is externally semi-active with
respect to $B$ if,
for the unique circuit $C = C^+ \sqcup C^{-} \subset B\cup\{j\}$,
both the minimal element of $C$ and $j$ belong to the same part $C^+$. This also explains our choice of terminology: in the `usual' notion of (external) activity, going all the way back to Tutte in the graph case and Crapo in the matroid case, the element inducing the circuit is required to \emph{be} smallest in order to be active. Here, we only require our external elements to \emph{be in the same class with} the smallest element of their circuits, and call them semi-active when they are.
\end{remark}

\subsection{Volume form and unimodular arrangements} Let  $\Vol$ be the volume form on $V\cong \R^d$ such that we may identify $V$ with $\R^d$ so that $\Vol$ is the
usual volume form on $\R^d$, and the vectors  $a_i$ become row $d$-vectors with
$$\Vol(\Pi_B)= \left|{\rm det}[a_{i_1}^T\mid \cdots \mid a_{i_d}^T]\right|
$$ for any base  
 $B=\{i_1, \ldots, i_d\}$ of $M[A]$ and the induced parallelepiped $\Pi_B:=\sum_{{i_j} \in B}[0,a_{i_j}]$. The arrangement  $a_1, \ldots, a_N \in V\cong \R^d$ is \textbf{unimodular} if for every base $B$ of $M[A]$ we have that $\Vol(\Pi_B)=1$. 

 Note that a full dimensional vector arrangement in $\R^d$  
 is {unimodular} if its vectors are integral and for any subset of vectors  $S$ that form a basis of $\R^d$, the determinant of the $d \times d$ matrix that contains the transposes of the vectors of $S$ as its columns is $\pm 1$.  
We call the matrix $A=[a_1^T\mid \cdots\mid a_N^T]$,    $a_i \in \R^d$, $i \in [N]$, full dimensional, respectively unimodular, if the vector arrangement $(a_1, \ldots, a_N)$   is  full dimensional,  respectively unimodular. A matrix $A$ is \textbf{totally unimodular} if all its minors are $0$ or $\pm 1$.

\subsection{Definition and invariance of the polynomials $ f_{{A}}(t)$}

In  \cite{slicing} Li and Postnikov introduce the following polynomial, which is the central object of study in this paper:

  \begin{definition} \label{def:f_A} \cite{slicing}  Let $a_1, \ldots, a_N \in V\cong \R^d$ be a flat collection of vectors. Fix a generic vector $\rho \in \R^N$  with respect to $\h_{A}^{sec}$. Define the polynomial $f_{{A}}(t)$ by
  \begin{equation}\label{A} f_{{A}}(t)=\sum_{B} t^{\ext_{\rho}(B)} \Vol(\Pi_B),\end{equation} where the sum is over all bases $B$ of $M[A]$.  \end{definition}

\begin{theorem} \label{thm:f_A}\cite{slicing} The polynomial $f_{{A}}(t)$ does not depend on the choice of generic vector $\rho \in \R^N$.
\end{theorem}

\noindent \textit{Proof idea:}   Theorem \ref{thm:f_A} is proven in \cite{slicing}  by verifying that  when a generic vector $\rho \in \R^N$ with respect to $\h_{A}^{sec}$ crosses a hyperplane $H_C$ for some circuit $C$, the polynomial $f_{A}(t)$ remains unchanged. 
\qed

\begin{remark}
If we replace $\rho$ with $-\rho$, then $C^-$ and $C^+$ trade places for all circuits $C$. This yields $\Ext_{-\rho}(B)=[N]\setminus B\setminus\Ext_\rho(B)$ for all bases $B$. In particular because $|B|$ is the same for all bases, this implies that $f_A(t)$ is a palindromic polynomial for all flat arrangements $A$.
\end{remark}

\begin{remark}\label{rem:oriented} We note that $f_A$ in Definition \ref{def:f_A} depends on the oriented matroid $M[A]$. Moreover, if $A$ and $B$ are both flat unimodular matrices yielding isomorphic oriented matroids  $M[A]$ and $M[B]$, then $f_A=f_B$. We will sometimes signify this by denoting $f_A$ by $f_{M[A]}$ instead. \end{remark}

\section{Flat matrices arising from graphic and cographic matroids}
\label{sec:graph-cograph}

In this section we study when the matrix presentation $A$ of a graphic or cographic matroid is flat. This is necessary in order for the polynomial $f_{A}$ to be defined. 
  For a thorough introduction to matroid theory, consult \cite{oxley}. For further information on oriented matroids, we refer the reader to  \cite{Oriented-matroids}. 

\subsection{Flat graphic  matroids} \label{sec:graph} The \textbf{graphic matroid} $M_G$  of the connected graph $G$ has ground set the edges of $G$ and bases given by all the spanning trees of $G$. 
One way to realize $M_G$ as a vector matroid is by fixing a reference orientation of $G$, which we denote by $D$, and taking  the (signed) \textbf{incidence matrix} $I(D)$ of $D$ to represent $M_G$. That is, after identifying $V(G)$ with $[|V(G)|]=\{1,\ldots,|V(G)|\}$, label  the columns of $I(D)$ by the edges $e$ of $G$, and let the column vector 
labeled by $e$ be $({\bf e}_i-{\bf e}_j)^T$ if
the edge $e$ points from vertex $i$ to vertex $j$ in $D$ and the ${\bf e}_i$ are the corresponding unit standard coordinate vectors in $\R^{|\Ver(G)|}$. It is well known that  the incidence matrix of a directed graph  is totally unimodular, thus we have obtained a totally unimodular matrix presentation of $M_G$. See Figure \ref{fig:inc} for an example.

\begin{figure}
\centering
\begin{tikzpicture}[scale=.4] 
    \tikzstyle{o}=[circle,scale=0.6,fill,draw]
    	\node[o, label=below:{\small $v_1$}] (1) at (0, 0) {};
        \node[o, label=left:{\small $v_2$}] (2) at (-4, 1) {};
        \node[o, label=above:{\small $v_3$}] (3) at (-2,4) {};
        \node[o, label=above:{\small $v_4$}] (4) at (2, 4) {};
        \node[o, label=right:{\small $v_5$}] (5) at (4, 1) {};
            \path[->,>=Stealth]
    		(1) edge node [below] {\small $e_1$} (2)
    		(1) edge node [right] {\small $e_2$} (3)
            (3) edge node [above] {\small $e_3$} (4)
    		(1) edge node [below] {\small $e_4$} (5);
            \path[->,>=Stealth]
    		(2) edge node [left] {\small $e_5$} (3)
    		(4) edge node [right] {\small $e_6$} (1)
    		(4) edge node [right] {\small $e_7$} (5);
\node at (17,2.5) {\small $I(D)=\bordermatrix
{~&\text{\footnotesize $e_1$}&\text{\footnotesize $e_2$}&\text{\footnotesize $e_3$}&\text{\footnotesize $e_4$}&\text{\footnotesize $e_5$}&\text{\footnotesize $e_6$}&\text{\footnotesize $e_7$}\cr
\text{\footnotesize $v_1$}&1&1&0&1&0&-1&0\cr
\text{\footnotesize $v_2$}&-1&0&0&0&1&0&0\cr
\text{\footnotesize $v_3$}&0&-1&1&0&-1&0&0\cr
\text{\footnotesize $v_4$}&0&0&-1&0&0&1&1\cr
\text{\footnotesize $v_5$}&0&0&0&-1&0&0&-1\cr}$};
\end{tikzpicture}
\caption{On the left is a directed graph $D$. 
On the right is its signed incidence matrix $I(D)$.}
\label{fig:inc}
\end{figure}
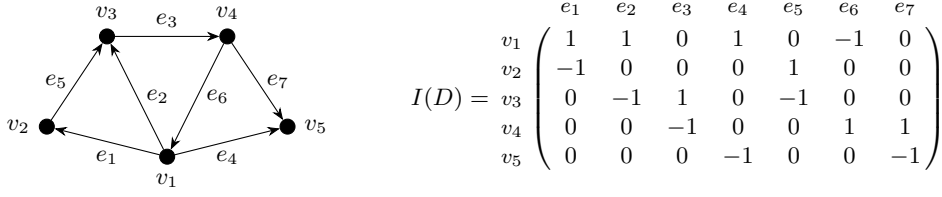

The above totally unimodular presentation of $M_G$ is very natural, but  the matrix $I(D)$ has $|\Ver(G)|$ many rows, yet is of rank  $|\Ver(G)|-1$ for a connected graph $G$. Indeed, the sum of 
the row vectors of $I(D)$ is $0$. One can  give a different totally unimodular presentation of $M_G$ via a full rank matrix of the form $[I_{|V(G)|-1}\mid M]$, where $I_{|\Ver(G)|-1}$ is the $(|V(G)|-1)\times (|\Ver(G)|-1)$ identity matrix, and $M$ is of size $(|\Ver(G)|-1) \times (|\Edg(G)|-|\Ver(G)|+1)$.  Moreover, the matrix $[I_{|\Ver(G)|-1}\mid M]$ constructed below has the exact same linear dependences among its columns as $I(D)$ does. (This is because it is in fact obtained from $I(D)$ by Gaussian elimination followed by deleting a trivial row.) As such, the oriented matroids $M[I(D)]$ and $M[I_{|\Ver(G)|-1}\mid M]$ are isomorphic and the polynomials $f_{I(D)}$ and $f_{[I_{|\Ver(G)|-1}\mid M]}$ are equal, as we pointed out in  
Remark \ref{rem:oriented}.

\begin{definition}
    \label{def:graph}
Fix a spanning tree $T$ of $G$, and choose edge labels so that  $\Edg(T)=\{e_1, \ldots, e_{|\Ver(G)|-1}\}$. 
This time, denote by ${\bf e}_i$ the unit coordinate vectors in $\R^{|\Ver(G)|-1}$ with one $1$ in position $i$ and all other coordinates $0$. 
Let the columns of 
the matrix $A_{D,T}$ be 
labeled by the edges of $G$. The first $|\Ver(G)|-1$ columns are labeled by $e_1, \ldots, e_{|\Ver(G)|-1}$ and they contain the column vectors ${\bf e}_1^T, \ldots, {\bf e}_{|\Ver(G)|-1}^T$, that is, we may write $A_{D,T}=[I_{|\Ver(G)|-1}\mid M]$. For each edge $e \in \Edg(G)\setminus \Edg(T)$, there is a unique cycle (called the \emph{fundamental cycle} of $e$ with respect to $T$) $C(T,e)$ contained in $T \cup \{e\}$. Traverse this cycle $C(T,e)$, whose edges in cyclic order are $e, e_{i_1}, \ldots, e_{i_k}$,  in the direction determined by $e \in D$ and assign $\alpha_{i_j}$ to be $-1$ or $1$ according to whether the edge $e_{i_j}$ is traversed in a way agreeing with its direction or disagreeing with its direction. Set the column vector of $M$ labeled by the edge $e$ to be $\sum_{j=1}^k \alpha_{i_j}{\bf e}_{i_j}^T.$ The thus obtained matrix $A_{D,T}=[I_{|\Ver(G)|-1} \mid M]$ represents $M_G$. 
We refer to the column of $A_{D,T}$ labeled by the edge $e$ as $v_e$.
\end{definition}

\begin{figure}
\centering
\begin{tikzpicture}[scale=.4] 
    \tikzstyle{o}=[circle,scale=0.6,fill,draw]
    	\node[o, label=below:{\small $v_1$}] (1) at (0, 0) {};
        \node[o, label=left:{\small $v_2$}] (2) at (-4, 1) {};
        \node[o, label=above:{\small $v_3$}] (3) at (-2,4) {};
        \node[o, label=above:{\small $v_4$}] (4) at (2, 4) {};
        \node[o, label=right:{\small $v_5$}] (5) at (4, 1) {};
            \path[->,>=stealth,ultra thick]
    		(1) edge node [below] {\small $e_1$} (2)
    		(1) edge node [right] {\small $e_2$} (3)
            (3) edge node [above] {\small $e_3$} (4)
    		(1) edge node [below] {\small $e_4$} (5);
            \path[->,>=Stealth]
    		(2) edge node [left] {\small $e_5$} (3)
    		(4) edge node [right] {\small $e_6$} (1)
    		(4) edge node [right] {\small $e_7$} (5);
\node at (17,2.5) {\small $A_{D,T}=
\begin{pNiceMatrix}
[first-row,first-col]
\CodeBefore
       \columncolor{red!15}{5-7}
     \Body
     &\text{\footnotesize $e_1$}&\text{\footnotesize $e_2$}&\text{\footnotesize $e_3$}&\text{\footnotesize $e_4$}&\text{\footnotesize $e_5$}&\text{\footnotesize $e_6$}&\text{\footnotesize $e_7$}&\\
     \text{\footnotesize $e_1$}&\,1&0&0&0&-1&0&0&{}\\
     \text{\footnotesize $e_2$}&\,0&1&0&0&1&-1&-1&{}\\
     \text{\footnotesize $e_3$}&\,0&0&1&0&0&-1&-1&{}\\
     \text{\footnotesize $e_4$}&\,0&0&0&1&0&0&1&{}\\
\end{pNiceMatrix}
$};
\end{tikzpicture}
\caption{On the left is a directed graph $D$ along with a spanning tree $T$. On the right is the matrix $A_{D,T}$ of Definition \ref{def:graph}, containing the highlighted block $M$.}
\label{fig:graphical}
\end{figure}
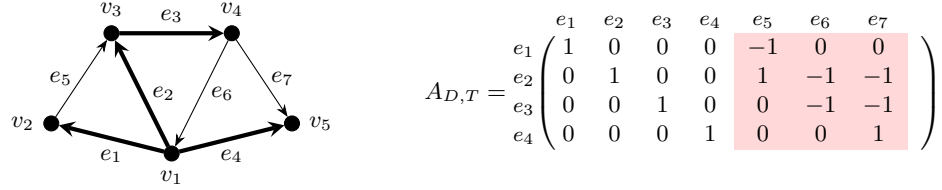

See Figure \ref{fig:graphical} for an example of this construction. 
We are interested in the case when the matrix $A_{D,T}$ is flat (unlike in Figure \ref{fig:graphical}), since that is when $f_{A_{D,T}}$ is defined. 
The following well-known characterization of bipartite graphs holds the key:

\begin{lemma} \label{lem:bip}
    A graph $G$ is bipartite if and only if every cycle of $G$ has even length.
\end{lemma}

\begin{lemma} \label{lem} If the  matrix $A_{D,T}$ is   flat then  $G$ must be a bipartite graph.
\end{lemma}

 \begin{proof} Note that if there is a linear form $h$ such that $h(v_e)$=1 for all edges $e$ in $G$, then since $A_{D,T}$ is of the form $[I_{|\Ver(G)|-1} \mid M]$, 
 in particular we have $h({\bf e}_i)=1$ for all $i=1,\ldots,|V(G)|-1$. Then as to the columns of the block $M$, the requirement becomes $\sum_{j=1}^k \alpha_{i_j}=1$, which tells us that 
 every cycle of the form $C(T,e)$, for $e \in \Edg(G)\setminus \Edg(T)$, must have an even number of edges. This in turn implies (by the fact that any cycle in $G$ is the modulo $2$ sum of the fundamental cycles of those of its edges that are not in $T$) that every cycle of $G$ must have an even number of edges. Therefore, by Lemma \ref{lem:bip}, $G$ must be bipartite. 
\end{proof}

\begin{lemma} \label{lem:?} Let $G$ be a bipartite graph  on the vertex bipartition $V_1\sqcup V_2$. Let $D$ be the reference orientation where all edges point from their vertex in $V_1$ to their vertex in $V_2$. Then the matrix $A_{D,T}=[I_{|\Ver(G)|-1} \mid M]$ from Definition \ref{def:graph} is flat, as is the incidence matrix $I(D)$.
    \end{lemma}
    
\begin{proof} A direct check show that for $h(x_1, \ldots, x_{|V(G)|-1})=x_1+\cdots+x_{|\Ver(G)|-1}$ and for every column $v_e$ of $A_{D,T}$ we have $h(v_e)$=1. The flatness of $I(D)$ is witnessed by the linear form $h(x_1, \ldots, x_{|\Ver(G)|})=\sum_{v \in V_1}x_v.$ 
\end{proof}

\begin{remark} \label{rem:graph}  Note that the oriented matroid $M[I(D)]$ is dependent on the digraph $D$ alone. Since   $M[I(D)]$ and $M[A_{D,T}]$ are isomorphic oriented matroids, we observe that the oriented matroid $M[A_{D,T}]$ does not depend on the choice of the tree $T$. For this reason, in the situation of Lemma \ref{lem:?}, --- that is, when $D$ is the reference orientation of the bipartite graph $G$ with vertex bipartition $V_1\sqcup V_2$ in which all edges are directed from their vertex in $V_1$ to their vertex in $V_2$  --- we will denote the corresponding polynomial $f_{A_{D,T}}$ by $f_{M_{D}}$ to emphasize that the polynomial depends on the oriented graphic matroid of $D$ alone. 
\end{remark}

\begin{remark}\label{rem:semibalanced}
It is easy to see that the flatness of $I(D)$, and the flatness of $A_{D,T}$ for any $T$, are equivalent. For a given bipartite graph $G$, the orientation of Lemma \ref{lem:?} is not the only one for which this is the case. According to \cite{arithm_symedgepoly}, a necessary and sufficient condition is that the vertices of $G$ be sorted to levels labeled by integers in such a way that the head of each edge is one level lower
than its tail. Then the linear extension, to $\R^{|\Ver(G)|}$, of such a `level' functional is a linear form that witnesses the flatness of $I(D)$. It is not hard to show, cf.\ \cite[Theorem 2.6]{semibalanced}, that another equivalent description of these orientations is that along any cycle of $G$, the edges that $D$ orients in the two opposite directions be equinumerous. K\'alm\'an and T\'othm\'er\'esz call these orientations \emph{semi-balanced}, and refer to the one in Lemma \ref{lem:?} as the \emph{standard orientation} of a bipartite graph. In the latter case there are only two levels and if we label them by $0$ and $1$, we end up with the functional in the proof of Lemma \ref{lem:?}.
Signed incidence matrices of semi-balanced, but not standard, orientations of bipartite graphs provide important test cases for the Question that we posed at the end of the Introduction.
\end{remark}

\subsection{Flat  cographic  matroids} \label{sec:matroiddefs} 
 The \textbf{cographic matroid} $M_G^*$ of a connected graph $G$ has ground set the edges of $G$ and  bases given by  the complements of the spanning trees of $G$. It is the matroid dual of the graphic matroid $M_G$. When the graph $G$ is planar, it is well known that $M_G^*=M_{G^*}$, where $G^*$ denotes the graph dual of a plane graph $G$. 

 In the cographic case we do not have an analogue of the incidence matrix presentation of graphic matroids, due to the fact that 
 the cycle space lacks 
 a canonical set of generators. Matrix presentations of the cographic matroid can be constructed by making some additional choices, such as fixing a spanning tree. Indeed, from the  totally unimodular presentation of the graphic matroid given in Definition \ref{def:graph}, one can obtain a totally unimodular matrix   realizing the cographic matroid via matroid duality \cite[Theorem 2.2.8]{oxley}. Instead of explaining the construction in the aforementioned way, we present it via feasible flows. This emphasizes the underlying fact that the fundamental cycles of any fixed tree provide a basis for the cycle space.

 Fix, then, a spanning tree $T$, as well as a reference orientation $D$ of the graph $G$.
 Recall that a \textbf{feasible flow} on the edges of the digraph $D$ is an assignment of real numbers $\{f_e\}_{e \in\Edg(D)}$ to the edges of $D$ so that the total netflow at each vertex (the sum of the flow values on the outgoing edges from the vertex minus the sum of the flows on the incoming edges to the vertex) 
is $0$. 
 Note here that the \emph{sum} of the netflows is always $0$, by which if the feasibility condition holds at any $|\Ver(G)|-1$ vertices, then it also holds at the last remaining vertex. 
 
 Recall also that the \textbf{fundamental cut}  in $G$ of an edge $d$ with respect to the spanning tree $T$ of $G$, where $T$ contains $d$,  denoted by 
  $C^*(T,d)$, is the cut connecting the vertices of the two components of $T$ with the edge $d$ removed. 
We can uniquely express any feasible  flow $\{f_e\}_{e\in\Edg(G)}$ in terms of  $\{f_e\}_{e \in\Edg(G)\setminus \Edg(T)}$  as in the following lemma: 

\begin{lemma} \label{lem:flow} Fix a spanning tree $T$ of a connected digraph $D$. Let  $\{f_e\}_{e \in\Edg(D)\setminus \Edg(T)}$ be arbitrary flow values assigned to 
the edges $e \in\Edg(D)\setminus \Edg(T)$. Then there is a unique way to extend this partial flow to a feasible flow on $D$. Namely, for an edge $d \in\Edg(T)$  its flow is determined as follows. The removal of $d$
partitions the vertex set of $T$, and therefore that of $D$, into $V_1 \sqcup V_2$. Then we have
\begin{equation} \label{eq:fd}f_d=\sum_{e\text{ opposite to } d}f_e\;-\sum_{e\text{ parallel to } d} f_e,\end{equation}  where the first  summation is over the edges of the fundamental cut $C^*(T,d)$
which point in the opposite direction between $V_1$ and $V_2$ to that of $d$, whereas the second summation is over the edges of the cut 
which point in the same direction between $V_1$ and $V_2$ as $d$.
\end{lemma}

\begin{proof} Since the flow over any cut in $D$ must be $0$,  it is clear that if there exists a flow extending the values $\{f_e\}_{e \in\Edg(D)\setminus \Edg(T)}$ to a feasible flow on $D$, then it must take the flow values given in \eqref{eq:fd}. Thus to prove the lemma it suffices to show that we can always extend 
the partial flow $\{f_e\}_{e \in\Edg(D)\setminus \Edg(T)}$ to a feasible flow on $D$. We fix an arbitrary root vertex $r\in\Ver(D)$ and start by determining the flows on the leaves of $T$ that are farthest away from $r$. After this, the remaining edges on which the flow is still not determined constitute a smaller tree. We continue by determining the flow on the `outer' leaves and repeat the process as long as there is any edge with no flow value. At the end we have a guaranteed netflow of $0$ at every vertex except $r$, but by that, the condition also automatically holds at $r$. Therefore the desired extension exists.
\end{proof}

Our presentation of the cographic matroid $M^*_G$ is as follows. See Figure \ref{fig:B_{D,T}} for an example. 

\begin{definition} \label{def:cograph} Fix a spanning tree $T$ of the connected graph $G$ and fix a reference orientation $D$ for $G$. Let $\{e_{|\Ver(G)|}, \ldots, e_{|\Edg(G)|}\}$ be the edges in $\Edg(G)\setminus \Edg(T)$.  We define the $(|\Edg(G)|-|\Edg(T)|)\times |\Edg(G)|$ matrix $B_{D,T}=[K \mid I_r]$, where $r=|\Edg(G)|-|\Ver(G)|+1$,  whose columns are labeled by the edges of $D$.  The last $r$ columns of $B_{D,T}$ are labeled by $e_{|\Ver(G)|}, \ldots, e_{|\Edg(G)|}$ and constitute the $r \times r$ identity matrix $I_r$.   For each edge $d \in \Edg(T)$, the fundamental cut $C^*(T,d)$ contains the edge $d$ and some edges from $\{e_{|V(G)|}, \ldots, e_{|\Edg(G)|}\}$. Let  $\alpha_{i}$, $i \in \{{|\Ver(G)|}, \ldots, {|\Edg(G)|}\}$, be $0$ if $e_i$ is not in the cut $C^*(T,d)$. Else let  $\alpha_{i}$ 
be $1$ or $-1$ 
according to whether the edge $e_{i}$ is opposite or parallel to $d$ in $C^*(T,d)$.  Set the column vector of the block $K$ labeled by the edge $d$ to be $w_d=\sum_{i=|\Ver(G)|}^{|\Edg(G)|} \alpha_{i}{\bf e}_i.$   We will refer to all columns of $B_{D,T}$, labeled by the edge $d$, by $w_d$. 
\end{definition}

\begin{figure}
\centering
\begin{tikzpicture}[scale=.4] 
    \tikzstyle{o}=[circle,scale=0.6,fill,draw]
    	\node[o, label=below:{\small $v_1$}] (1) at (0, 0) {};
        \node[o, label=left:{\small $v_2$}] (2) at (-4, 1) {};
        \node[o, label=above:{\small $v_3$}] (3) at (-2,4) {};
        \node[o, label=above:{\small $v_4$}] (4) at (2, 4) {};
        \node[o, label=right:{\small $v_5$}] (5) at (4, 1) {};
            \path[->,>=stealth,ultra thick]
    		(1) edge node [below] {\small $e_1$} (2)
    		(1) edge node [right] {\small $e_2$} (3)
            (3) edge node [above] {\small $e_3$} (4)
    		(1) edge node [below] {\small $e_4$} (5);
            \path[->,>=Stealth]
    		(2) edge node [left] {\small $e_5$} (3)
    		(4) edge node [right] {\small $e_6$} (1)
    		(4) edge node [right] {\small $e_7$} (5);
\node at (17,2.5) {\small $B_{D,T}=
\begin{pNiceMatrix}
[first-row,first-col]
\CodeBefore
       \columncolor{blue!15}{2-5}
     \Body
     &&\text{\footnotesize $e_1$}&\text{\footnotesize $e_2$}&\text{\footnotesize $e_3$}&\text{\footnotesize $e_4$}&\text{\footnotesize $e_5$}&\text{\footnotesize $e_6$}&\text{\footnotesize $e_7$}\\
     \text{\footnotesize $e_5$}&&1&-1&0&0&1&0&0\\
     \text{\footnotesize $e_6$}&&0&1&1&0&0&1&0\\
     \text{\footnotesize $e_7$}&&0&1&1&-1&0&0&1\\
\end{pNiceMatrix}$};
\end{tikzpicture}
\caption{On the left is a directed graph $D$ along with a fixed spanning tree $T$. On the right is the matrix $B_{D,T}$ of Definition \ref{def:cograph}, with its block $K$ highlighted.}
\label{fig:B_{D,T}}
\end{figure}
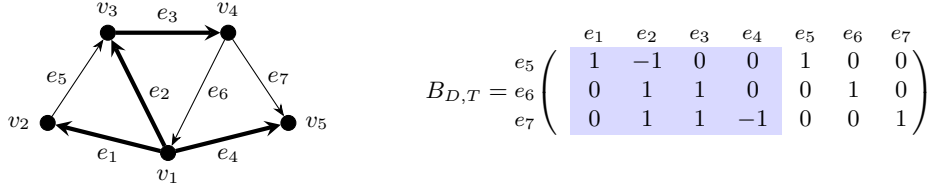

\begin{theorem}\label{thm:BDT} The matrix $B_{D,T}$ is totally unimodular and represents the cographic matroid $M^*_G$.
\end{theorem}

\begin{proof} The cographic matroid is well known to be isomorphic to the bond matroid of $G$. The circuits of the bond matroid of $G$ are the minimal edge cuts and they determine the matroid. Since a feasible flow  always has netflow $0$ over a cut, it follows that all minimal edge cuts of $G$ correspond to linearly dependent columns in  $B_{D,T}$. On the other hand, if a set of columns is dependent in $B_{D,T}$, it must contain a minimal cut. Indeed, assuming otherwise, the columns would correspond to a set of edges whose complement in $\Edg(D)$ contains a tree. However, by Lemma \ref{lem:flow} these are always linearly independent. 

To see that $B_{D,T}$ is totally unimodular we invoke matroid duality \cite[Theorem 2.2.8]{oxley}.  $B_{D,T}=[K\mid I_r]$ is totally unimodular if and only if the matrix representing its dual $[I_{|\Edg(G)|-r}\mid -K^T]$  is totally unimodular. The matrix $[I_{|\Edg(G)|-r}\mid -K^T]$ is the one   given in  Definition \ref{def:graph}. The latter is shown to be totally unimodular in \cite[p. 139]{oxley}. 
\end{proof}

 Next we investigate when the matrix $B_{D,T}$ is flat, since that is the case when $f_{B_{D,T}}$ is well-defined.
 The  answer relies on the following well-known characterization of Eulerian graphs: 

\begin{lemma} \label{lem:euler}
    A graph $G$ is Eulerian if and only if for every partition $V_1\sqcup V_2$ of the vertex set of $G$ the number of edges of $G$  that have one vertex in   $V_1$ and the other vertex in  $V_2$ is even.
\end{lemma}

\begin{lemma} Let $G$ be a  graph, $D$ a reference orientation of $G$ and $T$ an arbitrary spanning tree of $D$. If the matrix $B_{D,T}$ is flat,  then  $G$ is   Eulerian. \end{lemma}

 \begin{proof} Note that if there is a linear form $h$ such that $h(w_d)$=1 for all edges $d$ in $G$, then since  $B_{D,T}$ is of the form $[K\mid I_r]$, it must be that  $h({x_1, \ldots, x_r})=x_1+\cdots+x_r$ in order to make the last $r$ columns satisfy $h(w_d)=1$. Note also that in order for the columns of $K$ to also have all coordinates sum to $1$, every cut of the form $C^*(T,d)$, for $d \in \Edg(T)$, must have an even number of edges. This in particular implies (because any cut in $G$ is the modulo $2$ sum of the fundamental cuts of the edges that constitute the cut's intersection with $T$) that every  cut of  $G$ 
 must have an even number of edges. Therefore, by Lemma \ref{lem:euler}, $G$ must be Eulerian. 
\end{proof}

\begin{lemma} \label{lem:eu} Given an Eulerian graph $G$ and an Eulerian reference orientation $D$ of it, the matrix $B_{D,T}$ from Definition \ref{def:cograph} is flat.
    \end{lemma}
    
\begin{proof} A direct check shows that for $h({x_1, \ldots, x_r})=x_1+\cdots+x_r$, and every column $w_d$ of $B_{D,T}$, we have $h(w_d)=1$.
\end{proof}

\begin{remark} \label{rem:cograph} According to Lemma \ref{lem:eu}  the polynomial $f_{B_{D,T}}$ is defined for all Eulerian digraphs $D$. We note that the oriented matroid $M[B_{D,T}]$ is the dual of the oriented matroid $M[A_{D,T}]$, which by Remark \ref{rem:graph}
depends only on $D$ and not on $T$. Thus, we will write $f_{{M^*_D}}$ for $f_{B_{D,T}}$ since it depends only on the oriented cographic matroid of $D$. We prove in Theorem \ref{thm:pd} below that the polynomials  $f_{{M^*_D}}$, for $D$ Eulerian, are exactly the polynomials of Murasugi--Stoimenow given in equation \eqref{pdt}. 
\end{remark}

\subsection{The planar bipartite and Eulerian case} \label{sec:planar} 
Let $G$ be a plane bipartite graph and $H$ its planar dual Eulerian graph. We orient the edges of $G$ from one bipartition to another as in Lemma \ref{lem:?} and denote this orientation by $\vec{G}$. We pick the  Eulerian reference orientation $\vec{H}$  of  $H$    that is dual to that of $\vec{G}$, thereby making it into an alternating dimap.  As noted in Remarks \ref{rem:graph} and \ref{rem:cograph}, the polynomials $f_{A_{\vec{G},T}}$ and $f_{B_{\vec{H},T}}$ only depend on the oriented graphic and cographic matroids $M_{\vec{G}}$ and $M_{\vec{H}}^*$, respectively.  Moreover, since we picked $\vec{G}$ and $\vec{H}$ as duals, we have that $M_{\vec{G}}$ and $M_{\vec{H}}^*$ are isomorphic oriented matroids. It readily follows then that:

\begin{lemma} \label{lem:dual} Let $H$ be a plane Eulerian graph and $G$ its planar dual bipartite graph on the bipartition $V_1 \sqcup V_2$. Denote by $\vec{G}$ the orientation of $G$ where all edges point from a vertex of $V_1$ to a vertex of $V_2$. Let $\vec{H}$ be the dual  orientation of $H$, which is an alternating dimap. Then:
 $$f_{M_{\vec{H}}^*}(t)=f_{M_{\vec{G}}}(t).$$ 
\end{lemma}

\section{The Murasugi--Stoimenow  polynomials $P_D$ as $f_{M_D^*}$}
\label{sec:cographic}

 This section is devoted to the study of the polynomial $f_{{M_D^*}}$ in the case when $D$ is an Eulerian digraph (not necessarily planar). In other words, this is the flat oriented cographic case. We prove that this instance of $f_A$ coincides with the polynomials $P_D$ of Murasugi--Stoimenow (as defined in equation \eqref{pdt}), and thus in particular the Alexander polynomials of special alternating links (Definition/Theorem \ref{alexpoly}) also belong to this class  as stated in Theorem \ref{thm:intro-Alex} (see  Theorem \ref{thm:graphical} below). 
    
 \subsection{Flat cographic case} To understand $f_{M_D^*}$, where $D$ is an Eulerian digraph,  we interpret the external semi-activity associated to the flat cographic oriented matroid $M_D^*$ in graph theoretic language. Recall that a base of $M^*_D$ is the {complement} $\Edg(D)\setminus \Edg(T)$ of some spanning tree $T$ of $D$. Thus a {non-element} of that base is some edge $d\in \Edg(T)$, and the unique circuit of $M_D^*$  formed in the union of the  base $\Edg(D)\setminus \Edg(T)$  and $d$ is the fundamental cut $C^*(T,d)$.  The linear dependence between the vectors corresponding to the elements of $C^*(T,d)$  is essentially given by equation \eqref{eq:fd}; this follows if we take $B_{D,T}$ as our matrix representation of $M_D^*$. Thus, in light of Remark \ref{rem:order}, we can interpret the external semi-activity of an edge $d$ of $T$  with respect to the base $\Edg(D)\setminus \Edg(T)$ as follows:
 
 \begin{lemma}  \label{lem:ext}   Let $D$ be an Eulerian digraph with  an ordering $\sigma$ of the edges of $D$. 
 Let $\sigma$ define the generic vector $\rho$ as in Remark \ref{rem:order}.
 For a fixed spanning tree $T$ of  $D$, the removal of an edge $d\in \Edg(T)$ breaks the vertex set of $D$ into $V_1 \sqcup V_2$, where $V_1$ and $V_2$ are the vertex sets of the two components of the forest $F=(V, \Edg(T)\setminus \{d\})$. The edge $d$ is externally semi-active, for the base $\Edg(D)\setminus \Edg(T)$ of $M_D^*$ and with respect to $\rho$, if and only if the edge $d$ points from a vertex in $V_i$ to a vertex in $V_j$ and also the minimal edge of 
 $C^*(T,d)$
 points from a vertex in  $V_i$ to a vertex in $V_j$, where $\{i,j\}=\{1,2\}$. 
\end{lemma}
 
For the purposes of this section, we will denote the external semi-activity of a base $\Edg(D)\setminus \Edg(T)$ of $M_D^*$ by $\ext_{\sigma}(\Edg(D)\setminus \Edg(T))$, where $\sigma$ is an edge ordering of $D$ as in Lemma \ref{lem:ext}. 
The interpretation of the external semi-activity given in Lemma \ref{lem:ext} appears as a definition    in the work of T\'othm\'er\'esz \cite[Section 5]{T22}, where she calls our notion of the edge $d$ being  externally semi-active, with respect to $\Edg(D)\setminus \Edg(T)$,  the \textit{internal semi-activity} of $d$ with respect to $T$ instead.

\begin{lemma} \label{lem:sigma} Given an Eulerian digraph $D$, fix a root vertex $r$ of it. Fix also an Eulerian tour of $D$ 
and  label its edges consecutively with the integers $1,2,\ldots$ so that the edge $1$ has $r$ for its tail. With this edge ordering  $\sigma$ we have:
\[\left\{T \text{ a spanning tree of } D\mid\ext_{\sigma}(\Edg(D)\setminus \Edg(T))=k\right\} =\left\{T\mid T \text{ is a $k$-spanning tree of } D\right\}\] for all $k$. 
In particular, \quad
\(\left|\left\{T \text{ a spanning tree of } D\mid\ext_{\sigma}(\Edg(D)\setminus \Edg(T))=k\right\}\right|=c_k(D,r)\) \quad for all $k$.
\end{lemma}

\begin{proof}  The edge ordering $\sigma$ of $D$  guarantees that for any partition $\Ver(D)=V_1\sqcup V_2$ so that $r\in V_1$, the smallest edge of the associated cut points from $V_1$ to $V_2$. Thus, given a spanning tree $T$ of $D$ and an edge $d$ of $T$,  the circuit $C^*(T,d)$ of $M_D^*$ is partitioned so that $C^+$ is the set of its edges with tails in the shore containing $r$, and $C^-$ is the set of its edges with their heads there. Therefore, by Lemma \ref{lem:ext}, $d$ is externally semi-active with respect to $\Edg(D)\setminus\Edg(T)$, in the ordering $\sigma$, if and only if the tail of $d$ and the root $r$ belong to the same connected component of $T-d$. In other words, the condition is that $d$ points away from $r$ within $T$, and these are exactly the edges one needs to reverse in order to turn $T$ into an oriented spanning tree. We conclude that $T$ is a $k$-spanning tree for $D$ and $r$ if and only if the external semi-activity of $T$ is $k$.
\end{proof}

Now we are ready to relate $P_{{D}}$ from equation \eqref{pdt} to $f_{M_{D}^*}$:

\begin{theorem} \label{thm:pd} For an arbitrary Eulerian digraph $D$ we have
\begin{equation} \label{eq:int} P_{{D}}(t)=f_{M_{D}^*}(t).\end{equation} \end{theorem}

 \begin{proof} Fix the same ordering $\sigma$ of the edges of $D$ as in Lemma \ref{lem:sigma}.
  Since the set of vectors associated to any base of the oriented cographic matroid $M_D^*$ is unimodular, cf.\ Theorem \ref{thm:BDT}, Definition \ref{def:f_A} 
  becomes: 
 \[f_{M_D^*}(t)=\sum_{T} 
 t^{\ext_{\sigma}(\Edg(D)\setminus \Edg(T))},\]
where the sum is over all spanning trees of $D$. Since in Lemma \ref{lem:sigma}
we proved that 
\[\left|\left\{T \text{ a spanning tree of } D \mid\ext_{\sigma}(\Edg(D)\setminus \Edg(T))=k\right\}\right|=c_k(D,r)\] 
for all $k$, the theorem follows.
\end{proof}

 \subsection{The Alexander polynomials of special alternating links} 
Theorem \ref{thm:pd} and Lemma \ref{lem:dual} readily imply:
 
\begin{corollary} \label{cor1} Let $G$ be  plane bipartite graph with orientation $\vec{G}$ where the edges are directed from one bipartition to the other. Let $H$ be its planar Eulerian dual with orientation $\vec{H}$ dual to $\vec{G}$, thereby an alternating dimap. Then:
 $$P_{\vec{H}}(t)=f_{M_{\vec{H}}^*}(t)=f_{M_{\vec{G}}}(t).$$ 
\end{corollary}

   Note that Theorem \ref{alexpoly} and Corollary \ref{cor1} imply Theorem \ref{thm:intro-Alex}:
    
   \begin{theorem} \label{cor:bip} Let $H$ be a plane Eulerian graph with an Eulerian orientation ${\vec{H}}$ that is an alternating dimap, and $G$ its planar dual bipartite graph with dual orientation $\vec{G}$. 
   Then the Alexander polynomial $\Delta_{L_{G}}(-t)$, for the special alternating link  $L_{G}$ associated to $G$, equals $f_{M_{\vec{G}}}(t)$, and dually, $f_{M_{\vec{H}}^*}(t)$:
 \begin{equation} \label{eq:??} \Delta_{L_{G}}(-t)\sim f_{M_{\vec{G}}}(t)=f_{M_{\vec{H}}^*}(t).
\end{equation}
\end{theorem}

\section{$f_{A}$ as a refined count of integer points of polytopes}
\label{sec:intpoint}

In this section we provide an interpretation of the coefficients of  $f_{A}$ as a refined count of integer points of polytopes. Throughout this section, in  addition to the flat condition, we assume that the columns of the matrix $A$ form a  unimodular arrangement. Vector arrangements from both graphic and cographic matroids fall under this case.  We will utilize the results of this section when we prove log-concavity results for the polynomial  ~$f_{A}$ in Section \ref{sec:graphic}. 
 

\subsection{Zonotopes, trimming, and tiling} The  \textbf{zonotope} $Z_{A}\subset \R^{k}$, associated to the matrix $A =[a_1^T \mid \cdots \mid a_N^T]$ where $a_i\in \R^{k}$ for all $i \in [N]$,  is defined as the Minkowski sum \begin{equation}Z_{A}=\sum_{i=1}^N[0, a_i]=\left\{\sum_{i=1}^N t_ia_i\,\middle|\, 0\leq t_i\leq 1, i \in [N]\right\}.\end{equation} Throughout this section we denote the dimension of $Z_A$ by $d$, which is less than or equal to $k$.

For a fixed vector $l \in \R^k$, we define the \textbf{$l$-trimmed zonotope $Z_{A,l}^-$} as follows:
 \begin{equation} \label{eq:trim} Z_{A,l}^-=\conv(x \in Z_{A}\cap \Z^k \mid x+\epsilon l \in Z_{A} \text{ for some } \epsilon>0).\end{equation}

We note that the notion of $l$-trimming for a  fixed vector $l \in \R^k$ works more broadly for any polytope, not just a zonotope. We will utilize it for zonotopes only in this paper. 

\medskip

Zonotopes have beautiful combinatorial zonotopal tilings.  
 For a base $B=\{i_1, \ldots, i_d\}$ of $M[A]$, recall that we earlier used the notation $\Pi_B=\sum_{j=1}^d[0, a_{i_j}]$. This special case of the zonotope construction gives a parallelepiped because of linear independence.

\begin{theorem} \label{thm:tiling} \cite{slicing} 
For a generic vector $\rho \in \R^N$, the collection of parallelepipeds $$\Pi_B+\sum_{j\in \Ext_{\rho}(B)} a_j,$$
where $B$ ranges over all bases of the matroid $M[A]$, forms a 
tiling of $Z_{A}$.
\end{theorem}

\begin{proof}  Theorem \ref{thm:tiling} follows readily from the construction of the involved objects.\end{proof} 

\subsection{Levels of a trimmed zonotope} This section examines how the tiling of the zonotope  $Z_{A}$ from Theorem \ref{thm:tiling} gives us a handle on the integer points of the $l$-trimmed zonotope $Z_{A,l}^-$ under the right conditions.

 \begin{definition} \label{def:adm} Fix $m \in \{0,1,2,\ldots, d\}$. A vector $l \in \R^k$ is $m$-admissible  for  $A =[a_1^T \mid \cdots \mid a_N^T]$, $a_i\in \R^k$, $i \in [N]$,  if for every base $B=\{i_1, \ldots, i_d\}$ of $M[A]$ we can express $l=\sum_{j=1}^d \alpha_j {a_{i_j}}$ so that exactly $m$ coordinates of $(\alpha_1, \ldots, \alpha_d)$ are positive and exactly $(d-m)$ are negative.  
 \end{definition}  
 
\begin{lemma} \label{lem:par} Given a base $B=\{i_1, \ldots, i_d\}$ of $M[A]$, let  $l=c_1 a_{i_1}+\dots+ c_m a_{i_m}-(c_{m+1} a_{i_{m+1}}+\dots+c_{d} a_{i_{d}})$ be an $m$-admissible vector, where $c_i>0$ for all $i \in [d]$. The $l$-trimming of $\Pi_B$ is the unique vertex of $\Pi_B$ 
given by $a_{i_{m+1}}+\dots+a_{i_d}$. 
\end{lemma}
 \begin{proof} Readily follows by inspection.\end{proof}

 \begin{definition} \label{def:level} Fix a flat  matrix $C$, where the  linear form $h$ witnesses the flatness of $C$, that is, $h(c_i)=1$ for all columns $c_i^T$ of $C$.  An integer point $p \in Z_{C}$ is said to be on level $z$ if $h(c)=z$. \end{definition}
 

  \begin{theorem} \label{thm:f} Let $A =[a_1^T \mid \dots \mid a_N^T]$, $a_i\in \R^k$, $i \in [N]$, be a matrix arising from a  flat, unimodular arrangement, such that $Z_A$ is $d$-dimensional.  If $l\in \R^k$ is an $m$-admissible vector for  $A$ and $h$ is the   linear form such that $h(a_i)=1$ for all $i \in [N]$, then:
 $$\sum_{p \in Z_{A,l}^-\cap\Z^k}t^{h(p)}=f_{A}(t)\cdot t^{d-m}.$$
 \end{theorem}

 \begin{proof} By Lemma \ref{lem:par} there is a  unique 
 point in the $l$-trimmed parallelepiped $\Pi_B$, namely, a vertex on level $d-m$. Hence, once we fix $\rho$ and thus a zonotopal tiling of $Z_{A}$ as in  Theorem \ref{thm:tiling}, each tile $\Pi_B+\sum_{j\in \Ext_{\rho}(B)} a_j$ contributes a level $\ext_{\rho}(B)+d-m$ integer point to $Z_{A,l}^-$. Furthermore, these are all the integer points that $Z_{A,l}^-$ contains.  This follows since  for a dissection of any   polytope $P$ into top dimensional  polytopes $P_1, \ldots, P_z$, one readily proves that the set of integer points that generate the $l$-trimming of $P$ is the union of the same sets for the pieces
 $P_1, \ldots, P_z$. \end{proof}  
 
  \begin{corollary}  If $l\in \R^k$ is an $m$-admissible vector for  $A =[a_1^T \mid \dots \mid a_N^T]$, $a_i\in \R^k$, $i \in [N]$, then the number of integer points of $Z_{A,l}^-$ is the volume of $Z_{A}$.
 \end{corollary}

The topic of the next section is an application  of Theorem \ref{thm:f}: When $l\in \R^k$ is an $m$-admissible vector for  $A$, and  $Z_{A,l}^-$ is a generalized permutahedron, then we can apply the theory of Lorentzian polynomials to obtain the log-concavity of the coefficients of $f_{A}$.

\section{Log-concavity of $f_A$ in the flat graphic case}
\label{sec:graphic}

In this section we prove Theorem \ref{thm:intro-graphical} (see Theorem \ref{thm:graphical} below). We accomplish this by exhibiting an admissible vector in the flat oriented graphical case, showing that the corresponding vector trimmed zonotope is a generalized permutahedron, and applying the theory of Lorentzian polynomials to the latter.

\subsection{An admissible vector in the flat graphical case} As discussed in Section  \ref{sec:graph}, one instance of a flat vector arrangement arises from a connected   bipartite graph $G$  with vertex bipartition $V=\{1,\dots,m\}\sqcup \{\overline{1},\dots,\overline{n}\}$. We denote by $\vec{G}$ the orientation of $G$ where every edge is directed from $V_1=\{1,\dots,m\}$ to $V_2=\{\overline{1},\dots,\overline{n}\}$. 
 
 In order to utilize Theorem \ref{thm:f} for $M_{\vec{G}}$, we need to specify an admissible vector for $I(\vec{G})$. We do so in the following lemma.  Recall that the incidence matrix $I(\vec{G})$ is not full rank, rather its columns span the subspace $W$ of $\R^{m+n}$ 
 cut out by the equation  $\sum_{i=1}^m x_i+\sum_{j=1}^n x_{\bar{j}}=0$. The space $W$ can be identified with $\R^{m+n-1}$ via the projection $p(x_1, \ldots, x_m,x_{\overline1},\ldots,x_{\overline n})=(x_1, \ldots, x_m,x_{\overline1},\ldots,x_{\overline{n-1}})$. On $W$ we let the volume form $\Vol$ be the one induced by the standard volume form on  $\R^{m+n-1}$.  In the notation of the previous section, we have $k=m+n$ and $d=m+n-1$.

\begin{lemma} \label{flows}
 Let $G$ be a connected bipartite graph with vertex bipartition $V=\{1,\dots,m\}\sqcup \{\overline{1},\dots,\overline{n}\}$.  Let $l=(l_1, \ldots, l_m, l_{\bar{1}}, \ldots, l_{\bar{n}}) \in \Z^{m+n}$ be an integer vector satisfying: 
 \begin{enumerate}
 \item  $\sum_{i=1}^m l_i+\sum_{j=1}^n l_{\bar{j}}=0$,
 \item $l$ has a unique negative coordinate $l_{\bar{j}} < 0$ for some $j \in [n]$, 
 \item all coordinates other than the unique  negative coordinate  are positive.  
 \end{enumerate}
 Then the vector $l$ is $m$-admissible for
the signed incidence matrix $I(\vec{G})$.  
\end{lemma}

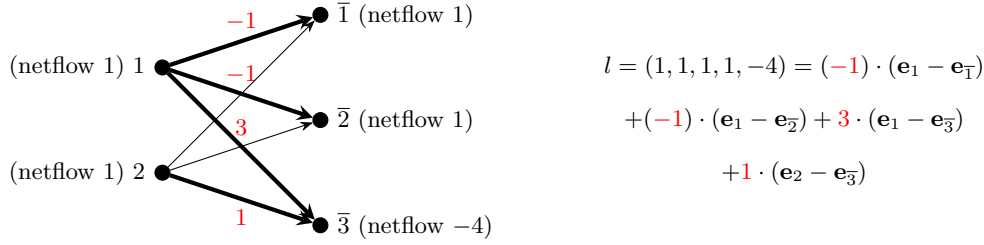
\begin{figure}
\centering
\begin{tikzpicture}[scale=.7] 
    \tikzstyle{o}=[circle,scale=0.6,fill,draw]
    	\node[o, label=left:{\small (netflow $1$) $1$}] (1) at (0, 4) {};
        \node[o, label=left:{\small (netflow $1$) $2$}] (2) at (0, 2) {};
        \node[o, label=right:{\small $\overline1$ (netflow $1$)}] (3) at (3,5) {};
        \node[o, label=right:{\small $\overline2$ (netflow $1$)}] (4) at (3, 3) {};
        \node[o, label=right:{\small $\overline3$ (netflow $-4$)}] (5) at (3, 1) {};
        \path[->,>=stealth,ultra thick]
    	(2) edge node [below] {\small \color{red} $1$} (5)
    	(1) edge node [above] {\small \color{red} $-1$} (3)
        (1) edge node [above] {\small \color{red} $-1$} (4)
    	(1) edge node [above] {\small \color{red} $3$} (5);
            \path[->,>=Stealth]
    		(2) edge node [left] {} (3)
    		(2) edge node [right] {} (4);
\node at (12,4) {\small $l=(1,1,1,1,-4)=({\color{red}-1})\cdot({\bf e}_{1}-{\bf e}_{\overline1})$};
\node at (12,3) {\small $+({\color{red}-1})\cdot({\bf e}_1-{\bf e}_{\overline2})+{\color{red}3}\cdot({\bf e}_1-{\bf e}_{\overline3})$};
\node at (12,2) {\small $+{\color{red}1}\cdot({\bf e}_2-{\bf e}_{\overline3})$};
\end{tikzpicture}
\caption{The flow values on  the edges of $\vec{G}$ with netflow vector $l$ coincide with the coefficients $\alpha_j$ from Definition  \ref{def:adm}.}
\label{fig:lem}
\end{figure}

\begin{proof} 
The bases of $M_{\vec{G}}$ are the spanning trees of $G$ and expressing some vector $l$ in terms of the vectors corresponding to a base (as in Definition \ref{def:adm}) means assigning flow values to each edge of the tree so that the netflow at each vertex is as prescribed by $l$.
That is, to say that a vector $l$ is $m$-admissible 
for $I(\vec{G})$ is to say that if we set the netflows of the vertices 
$\{1,\dots,m\}\sqcup \{\overline{1},\dots,\overline{n}\}$ to be the corresponding coordinates of $l$, and we fix any spanning tree $T$ of $G$, then exactly $n-1$ 
of the flows on the edges of $T$ will be negative and exactly $m$ 
 of the flows on the edges of $T$ will be positive. See Figure \ref{fig:lem} for an example. 

Now, for any vector $l$ as given in the Lemma, note that there are at least $m$ edges with positive 
flows, since all netflows on $\{1,\dots,m\}$ are positive, and there are at least $n-1$ flows that are negative
 because all but one of the netflows on $ \{\overline{1},\dots,\overline{n}\}$ are positive. However, there are exactly $m+n-1$ edges of any spanning tree of  $G$, and there is always a (unique) flow that we can put on it for any given netflow vector. Therefore, we indeed have exactly $m$ positive 
flows and $n-1$ negative 
flows. \end{proof}

 \begin{theorem} \label{cor:graphical} For a connected   bipartite graph $G$ with vertex bipartition $V=\{1,\dots,m\}\sqcup \{\overline{1},\dots,\overline{n}\}$, reference orientation $\vec{G}$ from $\{1,\dots,m\}$ to $\{\overline{1},\dots,\overline{n}\}$, and vector $l$ as in Lemma  \ref{flows}, we have 
  \begin{equation} \label{eq:}
  \sum_{p \in Z_{I(\vec{G}),l}^-\cap\Z^{m+n}}t^{h(p)}=f_{M_{\vec{G}}}(t)\cdot t^{n-1},
\end{equation} 
  where $h(x_1, \ldots, x_{m},x_{\overline1},\ldots,x_{\overline{n}})=x_1+\cdots+x_m$.
  In particular, the coefficients of $f_{M_{\vec{G}}}(t)$ count the numbers of integer points of $Z_{I(\vec{G}),l}^-$ falling on parallel hyperplanes of the form $x_1+\ldots+x_m=c$ 
  for integer constants $c$. 
 \end{theorem}

 \begin{proof} Recall that $M_{\vec{G}}$ is isomorphic to $M[I(\vec{G})]$, and the linear form $h$ witnessing the flatness of $I(\vec{G})$ is the one given in the Theorem. With this in mind, \eqref{eq:}  follows from Theorem \ref{thm:f} and Lemma \ref{flows}. 
 \end{proof}

\subsection{The vector trimmed zonotope is a generalized permutahedron} The notion of a  \textbf{trimmed zonotope}  $Z_{G}^-$, for any connected graph $G$, is defined in \cite{P05} as follows:
 \begin{equation}Z_{G}^-=\{x \in \R^{|\Ver(G)|} \mid x+ \Delta_{\Ver(G)}\subset Z_{I(\vec{G})}\},\end{equation}
 where $\Delta_{\Ver(G)}=\conv({\bf e}_i \mid i \in [|\Ver(G)|])$ is the standard simplex in $\R^{|\Ver(G)|}$. Here, changing the reference orientation $\vec{G}$ only changes $Z_{I(\vec{G})}$ and $Z_{G}^-$ by a parallel translation.
 The trimmed zonotope $Z_{G}^-$ is a special type of  generalized permutahedron \cite{P05}. Next we prove that our vector trimmed zonotope $Z_{I(\vec{G}),l}^-$, from Theorem \ref{cor:graphical}, also belongs to this class. 

\begin{lemma} \label{lem:trim}
 For a connected bipartite graph $G$ and vector $l$ as in Lemma  \ref{flows}, we have that $Z_{G}^-$ is a parallel translate of $Z_{I(\vec{G}),l}^-$.  
 \end{lemma}
 
 \begin{proof} Note  that $\Delta_{[m+n]}=\conv({\bf e}_i \mid i \in [m+n])$ is a parallel translate of the simplex $\Delta_{[m+n]}^j=\conv({\bf e}_i-{\bf e}_j \mid i \in [m+n])$, for any fixed $j \in [m+n]$. As such, $Z_{G}^-$  is a  parallel translate of $\{x \in \R^{m+n} \mid x+ \Delta^j_{[m+n]}\subset Z_{I(\vec{G})}\}.$ On the other hand note that if $l$, as in Lemma  \ref{flows}, has its unique negative coordinate in position $j$, then we have that  $Z_{I(\vec{G}),l}^-=\{x \in \R^{m+n} \mid x+ \Delta^j_{[m+n]}\subset Z_{I(\vec{G})}\}.$ Thus the statement follows.  \end{proof}  
 
 We remark that the previous Lemma holds for any reference orientation $\vec{G}$, whereas the Theorem below is particular to the orientation from one side of the bipartition to the other. 
 
\begin{theorem} \label{thm} The coefficients of $f_{M_{\vec{G}}}(t)$ count the numbers of integer points of $Z_{G}^-$ falling on parallel hyperplanes of the form $x_1+\dots+x_m=c$ for integer constants $c$. 
 \end{theorem}
 
 \begin{proof} Follows readily from Theorem \ref{cor:graphical} and Lemma \ref{lem:trim}.
 \end{proof}

   \subsection{Log-concavity of  $f_{M_{\vec{G}}}$ and a special case of Fox's conjecture} We prove Theorem \ref{thm:intro-graphical} in this section: 
 
 \begin{theorem} \label{thm:graphical} Let  $G$ be a connected bipartite graph with bipartition $V_1 \sqcup V_2$. Let $\vec{G}$ be the orientation of $G$ where all edges are directed from $V_1$ to $V_2$.  Then the coefficients of  $f_{M_{\vec{G}}}$ are log-concave. 
 \end{theorem}
 
 In order to obtain a proof of the above, we need the tool of Lorentzian polynomials. This section briefly summarizes everything that is needed about them; for an in-depth treatment and general definition of these polynomials, see the seminal work of Br\"and\'en and Huh \cite{bh2019}, where these polynomials were introduced. We will not give a full definition of the set of Lorentzian polynomials here, instead take the following results as a way of obtaining certain denormalized Lorentzian polynomials arising from integer point transforms of generalized permutahedra.

\begin{theorem}[{\cite[Theorem 3.10]{bh2019}}]
\label{3.10}
Let $J$ be the set of integer points of a generalized permutahedron $P$. Then the integer point transform   
$i_P({\bf x})=\sum_{{\bf \alpha} \in J}{\bf x}^{\bf \alpha}$ of $P$ is a denormalized Lorentzian polynomial. 
\end{theorem}
 
 \textbf{Generalized permutahedra}, defined as those polytopes whose edge directions are parallel to the root directions ${\bf e}_i-{\bf e}_j$, were extensively studied in \cite{P05}. Initially introduced by Edmonds \cite{Edmonds} in  1970 under a different name, they have appeared in the mathematical literature in many different guises. Integer points of integral generalized permutahedra are also known as M-convex sets and play a fundamental role in the theory of Lorentzian polynomials, which Br\"and\'en and Huh \cite{bh2019} introduced as the first comprehensive tool for proving log-concavity results. Lorentzian polynomials, under the name of completely log-concave polynomials, were also independently developed by Anari, Liu, Oveis Gharan, and   Vinzant  \cite{anari1, anari2, anari3}.   

\medskip

 Using the following lemma we can generate other denormalized Lorentzian polynomials from integer point transforms of generalized permutahedra:

\begin{lemma}\label{pak}\cite[Lemma 4.8]{pak} If $f(x_1, x_2, x_3, \ldots, x_n)\in \mathbb{R}_{\geq 0}[x_1, \ldots, x_n]$ is a denormalized Lorentzian polynomial, then  $f(x_1, x_1, x_3,  x_4, \ldots, x_n)$ is also a denormalized Lorentzian polynomial.
\end{lemma}

\begin{theorem} \label{thm:sections} Given a generalized permutahedron $P$, let $\ldots, 0, 0, 0, a_0, \ldots, a_k, 0, 0, 0, \ldots$ count the numbers of integer points of $P$ falling on parallel hyperplanes of the form $x_1+\dots+x_m=c$ for consecutive integer constants $c$,  where $a_0, a_k \neq 0$. Then the polynomial $\sum_{i=0}^k a_i q^i t^{k-i}$ is denormalized Lorentzian. 
\end{theorem}

\begin{proof} The statement follows from Theorem \ref{3.10}, Lemma \ref{pak} and \cite[Corollary 3.8]{hafner2023logconcavity}.\end{proof}

\begin{corollary} \label{cor} When homogenized, the polynomial $f_{M_{\vec{G}}}(t)$, for a connected bipartite graph $G$ and its orientation $\vec{G}$ from one side of the bipartition to the other, is denormalized Lorentzian.
\end{corollary}

\begin{proof} By Theorem \ref{thm}  the coefficients of $f_{M_{\vec{G}}}(t)$ count the numbers of integer points of $Z_{G}^-$ falling on parallel hyperplanes of the form $x_1+\dots+x_m=c$ for integer constants $c$. The trimmed zonotope $Z_{G}^-$ is a generalized permutahedron \cite{P05}.  Thus the result follows by Theorem \ref{thm:sections}.\end{proof}


The following key property of denormalized Lorentzian polynomials is what we need to complete the  proof of Theorem \ref{thm:graphical}:

\begin{proposition}[{\cite[Proposition 4.4]{bh2019}}]
\label{lem:log-concavity-of-coeffs}
If $f(\mathbf x) = \sum_\alpha c_\alpha \mathbf x^\alpha$ is a homogeneous polynomial in $n$ variables so that $f$ is denormalized Lorentzian, then for any $\alpha\in\mathbb{N}^n$ and any $i,j \in [n]$, the inequality
\[
c_\alpha^2 \geq c_{\alpha + e_i - e_j} c_{\alpha - e_i + e_j}
\]
holds.
\end{proposition}


\begin{proof}[Proof of Theorem \ref{thm:graphical}.] Follows readily from Corollary \ref{cor} and Proposition \ref{lem:log-concavity-of-coeffs}. \end{proof}


Theorems \ref{cor:bip} and \ref{thm:graphical}  imply the special case when the bipartite graph is planar:

\begin{corollary} \cite[Theorem 1.2]{hafner2023logconcavity} \label{cor:main} The coefficients of the Alexander polynomial $\Delta_L(-t)$ of a special alternating link $L$ form a  log-concave sequence with no internal zeros. In particular, they are  trapezoidal, proving Fox's conjecture in this case.
\end{corollary}

A natural open problem is whether $\Delta_L(-t)$, where the link $L$ is alternating but not necessarily special, may be obtained in the form $f_A(t)$ for some flat vector arrangement $A$, and whether the general case of Fox's conjecture can be settled with the help of such an interpretation.

\section{$f_{A}$ from totally positive matrices}
\label{sec:TP}

In this final section we prove Theorem \ref{thm:intro-totpos}, appearing as Corollary \ref{cor:C} below.  

\subsection{External semi-activity for matrices with positive 
maximal minors}
Let $A$ be a $d \times N$ matrix with all positive maximal minors.
The oriented matroid $M[A]$ is called the {\it alternating uniform
oriented matroid\/} of rank $d$ on $N$ elements. 
Bases of $M[A]$ are arbitrary $d$-element subsets $B$
of the ground set $[N]$.  Assume that $[N]$ is 
ordered linearly: $1<2<\cdots<N$.
Pick a generic vector $\rho\in\R^N$ such that 
$\rho_1 >> \rho_2 >> \cdots >> \rho_N > 0$ and orient the circuits 
of $M[A]$ as in Definition~\ref{def:ext}. These oriented circuits have 
the form $C= C^+\sqcup C^-$, where
$C=\{c_1<c_2<\ldots<c_{d+1}\}$ is an arbitrary $(d+1)$-element subset of $[N]$,
$C^+ = \{c_1,c_3,c_5,\dots,\}$, and $C^-=\{c_2,c_4,c_6,\dots,\}$.
The well-known fact that the signs in circuits of $M[A]$ alternate
follows from the identity for maximal minors of a $d\times (d+1)$ matrix:
$\sum_{i=1}^{d+1}(-1)^{i-1}{a_i} \Delta_{[d+1]\setminus \{i\}}=\vec{0}$,
where the $a_i$ are the column vectors of the matrix.

In this case, Definition~\ref{def:ext} of external semi-activity 
specializes as follows:
For a base $B=\{i_1<i_2<\ldots < i_d\}\subset [N]$, an 
element $j\in[N]\setminus B$ is externally  semi-active 
if $\#\{r \mid i_r < j\}$ is even.   Equivalently, 
the set $\Ext(B)$ of externally semi-active elements is
$$
\Ext(B) = \{j \mid j < i_1\} \cup \{j \mid i_2 < j < i_3\}
\cup \{j \mid i_4 < j < i_5\} \cup \cdots.
$$
We obtain the following expression for the external semi-activity
$\ext(B) := |\Ext(B)|$.

\begin{lemma} 
\label{lem:extB}
Let $A$ be a $d \times N$ matrix with all positive maximal minors.  Let
$B=\{i_1<i_2<\ldots < i_d\}$ be a base of the matroid  $M[A]$.  Let $i_0=0$ and
$i_{d+1}=N+1$. The external semi-activity of $B$ is 
$$
\ext(B)= (i_1 - i_0 - 1) + (i_3-i_2-1) + (i_5 - i_4 - 1)  + \cdots
= \sum_{k=0}^{2\lfloor d/2 \rfloor + 1}  (-1)^{k+1} i_k 
- \lfloor d/2 \rfloor -1. 
$$
\end{lemma}

\subsection{Flat max-positive matrices} 

Recall that, in order to define the polynomial $f_{A}$, 
we need to assume that $A$ is a flat matrix.
Let us
consider a special class of flat matrices $A$ that have all positive 
maximal minors.

\begin{definition}
\label{def:CA}
We say that a $d\times N$ matrix $A=(a_{ij})$ is 
{\it flat max-positive\/} if there exists a $(d-1)\times N$ 
matrix $C=(c_{ij})$ with all positive maximal minors such that $A$ is 
given by $a_{dj}=1$, for all $j \in [N]$,
and $a_{ij}=\sum_{N\geq j'\geq j}c_{ij'}$ for all $i \in [d-1], j \in [N]$. 
\end{definition}
 
Clearly, such a  matrix $A$ is flat because its last row consists of all $1$'s. It is also `max-positive' by the following lemma:

\begin{lemma} 
\label{lem:A-C}
Let $A$ be the 
matrix constructed from $C$ 
as in Definition~\ref{def:CA}.
The maximal minors of $A$ are expressed in terms of the maximal minors 
of $C$ as 
\begin{equation}\label{eq:minors}\Delta_{i_1, \ldots, i_d}(A)=\sum_{j_1,
\ldots, j_{d-1} \, \mid \, i_1\leq j_1<i_2\leq j_2<i_3\leq\cdots
\leq j_{d-1} < i_d} {\Delta}_{j_1, \ldots,
j_{d-1}}(C). 
\end{equation}
In particular, all maximal minors of $A$ are positive.
\end{lemma}

Note that the maximal minors $\Delta_{i_1,\dots,i_{d}}(A)$
do not depend on the last column of the matrix $C$.

\begin{proof}
Let $\tilde C$ be the $d\times N$ matrix obtained from $C$ by adding the 
$d$th row $(0,\dots,0,1)$.  Then $A = \tilde C \cdot J$, where
$J$ is the lower-triangular $N\times N$ matrix with all entries $1$
on and below the main diagonal.  
The claim of the lemma follows from the Cauchy--Binet formula.
We leave it as an exercise for the reader to calculate all minors of $J$.
\end{proof}

\begin{definition}
A matrix $A$ is called {\it totally positive\/} if all its minors
(of all sizes) are positive.
\end{definition}

\begin{lemma} \label{lem:TP} 
Any totally positive $d\times N$ matrix with the last row $(1,\dots,1)$
is flat max-positive.
More precisely, the totally positive $d\times N$ matrices with the last row
$(1,\dots,1)$ are exactly the matrices $A$ as in 
Definition~\ref{def:CA}, 
where $C$ is an arbitrary totally positive $(d-1)\times N$ matrix.
\end{lemma}

\begin{proof}
We will use the network parametrization of 
the space of totally positive matrices, see \cite{grass}.
Let $(G, \{x_e\})$ be a directed network 
(i.e., a directed graph $G$ with edge weights $x_e$)
such that $G$ is the rectangular $d\times N$ grid graph with horizontal
edges directed from right to left and vertical edges directed down,
with $d$ sources $1,\dots,d$ on the right and $N$ sinks $1',2',\dots,N'$
on the bottom, carrying positive edge weights $x_e>0$ for all horizontal edges
$e$,
and edge weights equal to $1$ for all vertical edges.
For such a network $(G,\{x_e\})$, construct a $d\times N$-matrix $A=(a_{ij})$,
where $a_{ij} = \sum_{P:i\to j'} \prod_{e\in P} x_e$, the weighted 
sum over all directed paths from source $i$ to sink $j'$.
Then the matrix $A$ is totally positive, and any totally positive matrix
has this form.

If the last row of $A$ is $(1,\dots,1)$, then the edge weights $x_e$ for 
all horizontal edges $e$ in the last row of the network are equal to $1$.
Let $C$ be the (totally positive) $(d-1)\times N$ matrix
given by the network obtained from $(G,\{x_e\})$ by removing the last row
of horizontal edges.
It is easy to see that the $d\times N$ matrix $A$ is expressed in terms of 
the $(d-1)\times N$ matrix $C$ as in Definition~\ref{def:CA}.
\end{proof}

\subsection{Box-positivity and trapezoidal
property of $f_{A}$ for flat totally positive matrices}

\begin{definition}  \label{def:box} A polynomial $f(q)$ of degree $D$ is called 
{\it $d$-box-positive\/} if $f(q)$ is a positive linear combination 
of products of $q$-numbers $[m_1]_q\, [m_2]_q\,\cdots [m_d]_q$
with $m_1+\cdots+m_d = D+d$ and $m_1,\dots,m_d\in\Z_{>0}$.
Recall that $q$-numbers are defined as 
$[m]_q := 1 + q + q^2 + \cdots + q^{m-1} = (1-q^m)/(1-q)$.
\end{definition}

\begin{lemma} \label{lem:box}
A $d$-box-positive polynomial $f(q)$ is trapezoidal and palindromic.
\end{lemma}

\begin{proof}
Note that any product of $q$-numbers $[m_1]_q\, [m_2]_q\,\cdots [m_d]_q$
is palindromic and log-concave, and thus trapezoidal.
This implies the claim because positive linear combinations 
of trapezoidal palindromic polynomials of the same degree are
also trapezoidal and palindromic.
\end{proof}

Note, however, that a $d$-box-positive polynomial may fail to be log-concave 
because log-concavity is not always preserved under addition of polynomials.

\begin{theorem}  \label{thm:A} 
For a totally positive $d\times N$ matrix $A$ with last row $(1,\dots,1)$
or, more generally, for a flat max-positive matrix $A$, 
the polynomial $f_A(q)$ is $d$-box positive.
Explicitly, we have 
\begin{equation} \label{eq:thm} f_A(q)
=\sum_{\{j_1< \cdots < j_{d-1}\}\subset [N-1]} 
{\Delta}_{j_1,\dots,j_{d-1}}(C) \, [j_1]_q [j_2 - j_1]_q [j_3 - j_2]_q \cdots [N - j_{d-1}]_q.
\end{equation}
\end{theorem}

\begin{proof} 
Using Lemmas~\ref{lem:extB} and~\ref{lem:A-C},
we get
$f_A(q) = \sum_{B=\{i_1<\cdots <i_d\}\subset [N]}
\Delta_{i_1,\dots,i_d}(A) \, q^{\ext(B)} = $
$$
\begin{array}{l} 
\displaystyle
=\sum_{i_1\leq j_1 < i_2 \leq j_2 <\cdots\leq j_{d-1}\leq i_d}
\Delta_{j_1,\dots,j_{d-1}}(C)\, 
q^{(i_1-1) + (i_3 - i_2-1) + (i_5-i_4-1) + \cdots}
\\[.1in]
\displaystyle
=\sum_{\{j_1<\dots<j_{d-1}\}\subset[N-1]} \Delta_{j_1,\dots,j_{d-1}}(C)\,
\left( \sum_{i_1\in[1,j_1]} q^{i_1-1}\right)
\left( \sum_{i_2\in[j_1+1,j_2]} q^{j_2-i_2}\right)
\left(\sum_{i_3\in[j_2+1,j_3]} q^{i_3 - j_2 -1}\right)\cdots,
\end{array}
$$
which gives the needed expression for $f_A(q)$.
\end{proof}

\begin{corollary}  \label{cor:C} For a totally positive matrix $A$ with last
row all $1$'s,  the polynomial $f_A(q)$ is trapezoidal. \end{corollary}

\begin{example} For $d=2$, a totally positive matrix with a last row of all $1$'s is of the form $A=\begin{pmatrix} a_1 & a_2 & \cdots & a_N \\
1 & 1 & \cdots & 1\\
\end{pmatrix} $, where $a_1>a_2>\cdots>a_N>0$. The corresponding $(d-1)\times N$ matrix is the row vector  
$C=(a_1-a_2,\dots,a_{N-1}-a_N, a_N)$. 
We obtain
$$
f_A(q)=\sum_{1\leq j\leq N-1}  c_j \, [j]_q \, [N-j]_q
$$ 
for positive constants $c_j = a_j - a_{j+1}$, $j\in[N-1]$.
Note that 
$$
[j]_q [N-j]_q=1+2t+\cdots+(j-1)t^{j-2}+j t^{j-1}+ j t^{j}+\cdots+ j t^{N-2-(j-1)}+(j-1)t^{N-2-(j-2)}+\cdots+1t^{N-2}.
$$
Thus the sequence of coefficients of $f_A(q)
= b_0 + b_1 q + \dots + b_{N-2} q^{N-2}$,
for a totally positive $2\times N$-matrix $A$ with last row all $1$'s, 
is an arbitrary positive {\it concave\/} palindromic sequence:
(1) $b_{i} > (b_{i-1}+b_{i+1})/2$ for all $i=1,\dots,N-3$; and
(2) $b_i = b_{N-2-i}>0$ for all $i=0,\dots,N-2$.
\end{example}
 
\bibliographystyle{alpha}
\bibliography{Arxiv-version-04-29-no-comments}

\end{document}